\documentclass{scrartcl}

\usepackage[utf8]{inputenc}
\usepackage[T1]{fontenc}
\usepackage{lmodern}
\usepackage{microtype}
\usepackage{pifont}
\usepackage[toc,page]{appendix}

\usepackage[english]{babel}

\usepackage[l2tabu, orthodox, experimental]{nag}

\usepackage{hyperref}
\hypersetup{unicode}

\usepackage{csquotes}

\usepackage{bm}

\usepackage{enumitem}
\setenumerate[1]{label=\textup{(\alph*)}}
\setenumerate[2]{label=\textup{(\roman*)}}

\usepackage{xspace}

\usepackage{makeidx}

\usepackage[dvipsnames]{xcolor}

\makeatletter
\g@addto@macro\bfseries{\boldmath}
\makeatother

\usepackage{todonotes}

\usepackage{amsmath, amsfonts, amssymb, amsthm, mathtools}
\usepackage{cancel}
\usepackage{tikz, pgf}
\usetikzlibrary{calc, cd, shapes,arrows.meta,shapes.geometric,positioning,fit}

\usepackage{cleveref}
\crefname{theorem}{Theorem}{Theorems}
\crefname{lemma}{Lemma}{Lemmata}
\crefname{corollary}{Corollary}{Corollaries}
\crefname{definition}{Definition}{Definitions}

\renewcommand*{\vec}[1]{{\mathbf{#1}}}

\newcommand*{\dotcup}{\mathbin{\dot\cup}}

\newcommand{\N}{\mathbb{N}}

\newcommand*{\FO}{\mathrm{FO}}

\newcommand{\cequiv}[1]{\equiv_{C_{#1}}}
\newcommand{\IFPC}{\mathrm{FP+C}}

\newcommand{\bpk}{\text{BP}_k}
\newcommand{\bp}[1]{\text{BP}_{#1}}

\newcommand{\subd}[1]{{#1}^{(s,<)}}
\newcommand{\atp}{\text{atp}}

\newcommand{\dimwl}{\dim_{\mathrm{WL}}}

\DeclareMathOperator{\rw}{rw}
\DeclareMathOperator{\tww}{tww}
\newcommand*{\red}{{\mathrm{red}}}
\DeclareMathOperator{\rdeg}{red-deg}

\DeclareMathOperator{\rk}{rk}

\newcommand*{\multil}[1][]{#1\{\mskip-3mu#1\{}
\newcommand*{\multir}[1][]{#1\}\mskip-3mu#1\}}

\theoremstyle{definition}
\newtheorem{definition}{Definition}[section]

\theoremstyle{plain}
\newtheorem{lemma}[definition]{Lemma}

\newtheorem{theorem}[definition]{Theorem}	
\newtheorem{corollary}[definition]{Corollary}

\newtheorem*{claim}{Claim}

\newenvironment{claimproof}[1][\proofname]{\begin{proof}[#1]}{\end{proof}}
\newenvironment{proof_sketch}{\renewcommand{\proofname}{Proof sketch.}\proof}{\endproof}

\makeatletter
\let\@vareps\varepsilon
\let\varepsilon\epsilon
\let\epsilon\@vareps
\makeatother

\renewcommand{\phi}{\varphi}

\title{Weisfeiler-Leman on graphs of small twin-width}
\author{Irene Heinrich\thanks{TU Darmstadt. The author has received funding from the European Research Council (ERC) under the European Union’s Horizon 2020 research and innovation programme (EngageS: grant agreement No.\ 820148).
	Views and opinions expressed are however those of the author(s) only and do not necessarily reflect those of the European Union or the European Research Council. Neither the European Union nor the granting authority can be held responsible for them.}\\[1mm]
	Klara Pakhomenko\thanks{Hasselt University. This author is supported by the Flanders AI Research Program.}\and Moritz Lichter\thanks{RWTH Aachen University. The author received funding from the European Research Council (ERC)
		under the European Union’s Horizon 2020 research and innovation programme (SymSim: grant
		agreement No. 101054974). Views and opinions expressed are however those of the author(s) only and
		do not necessarily reflect those of the European Union or the European Research Council. Neither
		the European Union nor the granting authority can be held responsible for them.}\\[1mm] Simon Raßmann\thanks{TU Darmstadt.}}
\makeatletter
\hypersetup{pdftitle="\@title"}
\makeatother

\begin{document}
\maketitle

\begin{abstract}
Twin-width is a graph parameter introduced in the context of first-order model checking, and has since become a central
parameter in algorithmic graph theory. While many algorithmic problems become easier on arbitrary classes of bounded twin-width, graph isomorphism on graphs of twin-width~\(4\) and above is as hard as the general isomorphism problem.
For each positive number $k$, the $k$-dimensional Weisfeiler-Leman algorithm is an iterative color refinement algorithm that encodes structural similarities and serves as a fundamental tool for distinguishing non-isomorphic graphs.
We show that the graph isomorphism problem for graphs of twin-width~\(1\)
can be solved by the purely combinatorial \(3\)-dimensional Weisfeiler-Leman algorithm, while 
there is no fixed $k$ such that the $k$-dimensional
Weisfeiler-Leman algorithm solves the graph isomorphism problem for graphs of twin-width~\(4\).

Moreover, we prove the conjecture of Bergougnoux, Gajarský, Guspiel, Hlinený, Pokrývka, and Sokolowski that stable graphs of twin-width~\(2\) have bounded rank-width.
This in particular implies that isomorphism of these graphs can be decided by a fixed dimension of the Weisfeiler-Leman algorithm.
\end{abstract}

\section{Introduction}
Twin-width is a graph parameter introduced in~\cite{tww1}, and has since become a central parameter in algorithmic graph theory
due to its significance in model checking:
First-order model checking on graphs becomes fixed-parameter tractable when witnesses
of small twin-width are provided. This generalizes similar results for classical graph classes like classes of bounded tree-width, bounded rank-width,
bounded genus, classes excluding a minor, or map graphs. Even more strongly, bounded twin-width exactly describes the boundary of tractable first-order model checking
on ordered graphs~\cite{tww4}, tournaments~\cite{tww_tournaments}, interval graphs~\cite{tww8}, and permutation graphs~\cite{tww1}.

Twin-width is based on repeatedly merging pairs of vertices with similar neighborhoods, while keeping the total error at every merged set of vertices bounded (we postpone a formal definition to \Cref{sec:preliminaries}). For example, since every cograph can be contracted to a single vertex by repeatedly merging pairs of twins, that is,
vertices with identical closed or open neighborhoods, cographs can be contracted without creating any errors, and thus have twin-width 0. In fact, cographs are exactly the graphs of twin-width~\(0\). In particular, graphs of twin-width~\(0\) can be recognized efficiently, and many algorithmic problems, including the graph isomorphism problem and even graph canonization, can be solved in polynomial time on this class
\cite{cograph_isomorphism}.

The recognition results also hold for twin-width~\(1\) graphs, where a polynomial-time recognition algorithm was given in~\cite{tww_polynomial_kernels}, which was later improved
to linear-time in~\cite{tww_one}. Moreover, in~\cite{tww_polynomial_kernels} it was also shown that graphs of twin-width~\(1\) have bounded rank-width, which
implies that isomorphism and canonization can be solved in polynomial time~\cite{rank-width_isomorphism, rank-width}.
In contrast to these positive results, recognition of graphs of twin-width~\(4\) is known to be NP-hard~\cite{twwleq4NPharda}.
Moreover, since the \(\Omega(\log(n))\)-subdivision of every graph has twin-width at most \(4\)~\cite{twwleq4NPharda},
the global structure of twin-width~\(4\) graphs is essentially unrestricted, and in particular,
graph isomorphism on the class of graphs of twin-width at most \(4\) is as hard as the isomorphism problem on general graphs. This leaves open the cases of twin-width~\(2\) and~\(3\), whose structural complexity is as of yet not well-understood, and where the complexity of both recognition and isomorphism is open.
Note that in contrast to hardness of isomorphism on graphs of twin-width~\(4\) or larger, isomorphism of tournaments parameterized by their twin-width is fixed-parameter tractable~\cite{tww_tournament_isomorphism},
and thus in particular polynomial-time solvable on every class of tournaments of bounded twin-width.

In this paper, we try to gain a better understanding of the combinatorial complexity of graphs
of small twin-width.
One natural framework for this is provided by the Weisfeiler-Leman algorithm,
which is a fundamental combinatorial approach to graph isomorphism that 
measures how much ‘distinguishing power’ is needed to tell graphs apart.
The Weisfeiler-Leman algorithm is a family of purely combinatorial polynomial-time algorithms which subsume most other combinatorial algorithms for the graph isomorphism problem.
The \(1\)-dimensional Weisfeiler-Leman algorithm, which is also called color refinement or naive vertex classification,
starts by coloring each vertex of the graph by its degree, and then repeatedly refines the color of each vertex by the multiset of colors of its neighbors,
until the partition induced by the coloring stabilizes. Then, the multiset of colors can be used as an isomorphism invariant for the graphs. Even this simple
heuristic identifies almost all graphs up to isomorphism~\cite{babai_erdos_selkow}, and in particular, identifies every forest. Moreover, since almost all graphs have twin-width linear in their order~\cite{twwpaley}, the color refinement algorithm can identify graphs of unbounded twin-width.

Generalizing color refinement, the \(k\)-dimensional Weisfeiler-Leman algorithm colors \(k\)-tuples of vertices instead of single vertices. The minimal dimension \(k\) such that the coloring computed by the \(k\)-dimensional Weisfeiler-Leman algorithm identifies every graph from some class \(\mathcal{C}\) up to isomorphism is called the \emph{Weisfeiler-Leman dimension} of \(\mathcal{C}\).
It turns out that many natural graph classes have bounded Weisfeiler-Leman dimension, which implies
that the \(k\)-dimensional Weisfeiler-Leman algorithm correctly decides isomorphism on this class.
This includes graph classes of bounded tree-width or bounded rank-width~\cite{rank-width}, planar graphs and graphs of bounded genus and most generally classes excluding some minor~\cite{grohe_minors}.
However, the class of all graphs has unbounded Weisfeiler-Leman dimension, as exemplified by 
the construction of so-called \emph{CFI graphs} by Cai, Fürer, and Immerman in their seminal paper~\cite{cfi}.
The \(k\)-dimensional Weisfeiler-Leman algorithm has a close correspondence to \((k+1)\)-variable first-order counting logic~\cite{cfi}. This connection can routinely be exploited to translate bounds on the Weisfeiler-Leman dimension into
results that graphs in a specific class can be definably canonized in fixed-point logic with counting \(\IFPC\).
Such definability results imply that \(\IFPC\) can express exactly the polynomial-time decidable
properties on these graphs. \(\IFPC\)-definable canonization is possible for all classes of bounded rank-width~\cite{rank-width} and classes excluding a minor~\cite{grohe_minors}.

\paragraph*{Our results.}
We investigate the Weisfeiler-Leman dimension of graphs of small twin-width.
First, we show that in general, the Weisfeiler-Leman dimension
of a graph cannot be bounded in terms of its twin-width.
Specifically, we show that graphs of twin-width~\(4\) have unbounded Weisfeiler-Leman dimension,
meaning that no fixed-dimensional Weisfeiler-Leman algorithm can correctly distinguish all pairs of non-isomorphic
graphs in this class. Our result is based on subdivisions of CFI graphs.

In contrast to this, we show that the \(3\)-dimensional Weisfeiler-Leman algorithm suffices to identify all graphs of twin-width at most \(1\),
by showing that the algorithm can emulate the recognition algorithm of~\cite{tww_polynomial_kernels}
in order to construct a canonical \(1\)-contraction sequence for the graph,
which can then be used to identify the graph.
While it is known that graphs of twin-width~\(1\) have bounded rank-width, and thus that some fixed-dimensional Weisfeiler-Leman algorithm correctly decides isomorphism, our argument provides the precise dimension \(3\) and is much more explicit. Moreover, we show how our result translates to a more efficient \(\IFPC\)-definable canonization algorithm for graphs of twin-width~\(1\).

Finally, we consider graphs of twin-width~\(2\).
It was shown in~\cite{sparse_tww2_bounded_tw} that sparse graphs of twin-width~\(2\) have bounded tree-width,
and conjectured that stable graphs (that is, graphs excluding some fixed semi-induced half-graph \(H_t\)) of twin-width~\(2\) have bounded rank-width.
We prove this conjecture by showing that graphs of twin-width~\(2\) without a semi-induced half-graph \(H_t\) have rank-width~\(O(t)\). To achieve this, we first classify bipartite graphs of twin-width~\(1\), and then leverage this classification to analyze contractions of red paths
in contraction sequences of width~\(2\).

Using the same techniques, we also show that the tree-width of graphs of twin-width~\(2\) without a \(K_{t,t}\)-subgraph is bounded by \(O(t)\),
thus improving the bound given in~\cite{sparse_tww2_bounded_tw}. In both of these cases, it follows that the \(O(t)\)-dimensional Weisfeiler-Leman algorithm
correctly decides isomorphism of these graphs.

\paragraph*{Further related work.}
Twin-width has been extensively studied since its introduction~\cite{tww1}, which among others established bounds in terms of boolean width (and thus rank-width and clique-width) and for minor-free classes. Further connections include tree-width~\cite{bounding-twwa, sparse_tww2_bounded_tw}, and topological parameters such as genus~\cite{tww_genus}. Sharp bounds are known for planar graphs~\cite{tww_planar},
and relations to other graph parameters have also been considered in restricted
graph classes~\cite{comparing_widths}.

For small twin-widths, recognition has already been studied:
While recognition is polynomial-time for twin-width at most 1~\cite{tww_polynomial_kernels, tww_one} but NP-hard for twin-width~\(4\)~\cite{twwleq4NPharda}, the complexity of twin-width~\(2\) and~\(3\) remains open.

This paper studies the relationship of twin-width
to the Weisfeiler-Leman dimension of a graph,
which provides a complementary complexity measure: many algorithmically tractable classes have bounded Weisfeiler-Leman dimension, including graphs of bounded tree-width~\cite{grohe_tw}, rank-width~\cite{rank-width}, planar graphs~\cite{grohe_planar, kiefer_planar}, and minor-free classes~\cite{grohe_minors}. However, the relationship between twin-width and Weisfeiler-Leman dimension was previously explored only for tournaments~\cite{tww_tournament_isomorphism}, where bounded twin-width yields tractable isomorphism despite unbounded Weisfeiler-Leman dimension.

Our study of this relationship for general graphs also resolves the conjecture from~\cite{sparse_tww2_bounded_tw} on rank-width bounds for stable graphs of twin-width~\(2\). Stable graphs of bounded twin-width have been studied independently in~\cite{stable_bounded_tww}.

\section{Preliminaries}\label{sec:preliminaries}
\newcommand{\vminor}{\preceq_{\mathrm{vtx}}}
\newcommand{\adj}{\operatorname{Adj}}
\newcommand{\F}{\mathbb{F}}
\newcommand{\mm}{\operatorname{mm}}

For a natural number $k$ we set $[k] \coloneqq \{1,2, \dots, k\}$.
The rank of a matrix $M$ is denoted by~$\rk(M)$.

\paragraph*{Graphs}
All graphs in this paper are finite and simple.
We say a graph is coconnected if its complement is connected.
For a graph \(G\), we write \(V(G)\) for its vertex set and \(E(G)\) for its edge set.
The \emph{order} of \(G\) is \(|V(G)|\). For a vertex set \(S\subseteq V(G)\),
we write \(G[S]\) for the subgraph induced by~$S$ in~\(G\). For two disjoint
sets \(L,R\subseteq V(G)\), we write \(G[L,R]\) for the \emph{bipartite graph with parts \(L\) and \(R\)
induced by \(G\)}. A bipartite graph \(H[A,B]\) is a \emph{semi-induced subgraph} of \(G\) if
there exists an embedding \(\iota\colon V(H)\to V(G)\) such that the induced map \(H[A,B]\to G[\iota(A),\iota(B)]\)
is an isomorphism. If \(H\) is not a semi-induced subgraph of \(G\), then~\(G\) is \emph{\(H\)-semi-free}.

We denote the neighborhood
of \(v\in V(G)\) in \(G\) by \(N_G(v)\) or just \(N(v)\) if \(G\) is clear from context.
We also write \(N_G^{=d}(v)\), \(N_G^{\leq d}(v)\), and \(N_G^{>d}(v)\) to denote the set of
vertices at distance exactly \(d\), at most \(d\), or greater than \(d\) from \(v\) in \(G\).

For a graph $G$ and a tuple $\vec{v} \in V(G)^k$,
the \emph{atomic type} $\atp_G(\vec{v})$ is the isomorphism type
of the ordered graph \(G[\vec{v}]\), i.e., two tuples \(\vec{v}\in V(G)^k\) and \(\vec{w}\in V(H)^k\)
are of the same atomic type if and only if \(v_i\mapsto w_i\) defines an isomorphism
between \(G[\vec{v}]\) and \(H[\vec{w}]\).

A colored graph $G^c$ is a tuple $(G,c)$
in which $G$ is a graph, and $c$ is a map
defined on $V(G)$, usually into the natural numbers.

\paragraph*{Twin-width}
A \emph{trigraph} $G$ is a graph with edges colored either red or black.
Graphs are interpreted as trigraphs by coloring each edge black.
For $v \in V(G)$ the \emph{red degree} $\rdeg_G(v)$ is the number of red edges incident to~$v$.
A \emph{red component} of a trigraph $G$ is a component of the graph obtained from $G$ by removing the set of all black edges.

Two disjoint vertex subsets $U$ and $W$ of $G$ are \emph{fully connected} if every 2-set \(\{u,w\}\) with \(u\in U\) and \(w\in W\) is a black edge. If no such pair is an edge of~$G$, then \(U\) and \(W\) are \emph{disconnected}. \(U\) and \(W\) are \emph{homogeneously connected} if they are either fully connected or disconnected.
Given a partition~\(\mathcal{P}\) of \(V(G)\), we define the \emph{quotient graph} \(G/\mathcal{P}\) to be the trigraph with $V(G/\mathcal{P}) = \mathcal{P}$ such that two parts $U$ and $W$ of $\mathcal{P}$ are
\begin{align*}
	&\text{joined via a black edge} &\text{ if $U$ and $W$ are fully connected,}\\
	&\text{not adjacent} &\text{ if $U$ and $W$ are disconnected, and}\\
	&\text{joined via a red edge} &\text{ otherwise.}
\end{align*}
A \emph{partition sequence} of an order-$n$ trigraph \(G\) is a sequence \(\mathcal{P}_n, \mathcal{P}_{n-1}, \dots,\mathcal{P}_1\) of partitions 
of \(V(G)\), where \(\mathcal{P}_n\) is the discrete partition and for each \(i \in [n-1]\) the partition \(\mathcal{P}_i\) is obtained
by replacing two distinct parts
$P$ and $Q$ of \(\mathcal{P}_{i+1}\) by~\(P\cup Q\) (we call this a \emph{merge} or a \emph{contraction}).
Equivalently, a partition sequence is given by the sequence of trigraphs~\(G/\mathcal{P}_i\)
or the sequence of pairs of vertices to contract in these trigraphs, both of which are called \emph{contraction sequences}.
As a shorthand to describe the contraction of one pair,
by $G/w,x$ we denote $G/\mathcal{P}$ where
$\mathcal{P} = \{\{w,x\}\} \cup \{ \{v\}: v \in V(G) \}$.
The \emph{width} of a contraction sequence (or its associated partition sequence) is the maximum red degree over all trigraphs~\(G/\mathcal{P}_i\). If a partition or contraction sequence has width at most \(k\), we also call it a \(k\)-partition sequence or \(k\)-contraction sequence respectively.
The \emph{twin-width} \(\tww(G)\) of \(G\) is the minimal width of a contraction sequence of $G$.
The twin-width of a colored graph is the twin-width
of the underlying uncolored graph.

The \emph{component twin-width} of~$G$ is the minimum over all contraction sequences
of the maximal order of red components in the appearing trigraphs. Component twin-width is functionally equivalent
to rank-width~\cite{tww6a} and differs by at most a constant factor from clique-width~\cite{ctww_vs_cw}.

\paragraph*{Ranks, rank-width, and well-linked sets}
Let \(G\) be a graph. Given two vertex sets \(A,B\subseteq V(G)\) the \emph{biadjacency matrix of \(A\) and \(B\)}, denoted by \(\adj_G(A,B)\), is
the \(|A|\times|B|\)-matrix over \(\F_2\) with rows indexed by vertices in \(A\) and columns are indexed by vertices in \(B\)  such that the \((a,b)\)-th entry is \(1\) if and only if \(ab\in E(G)\).
We set \(\rk_G(A,B)\coloneqq\rk(\adj_G(A,B))\).
The rank-function is subadditive: If \(A,B,C\subseteq V(G)\) are pairwise disjoint,
then \(\rk_G(A,B\cup C)\leq\rk_G(A,B)+\rk_G(A,C)\).

Let \(A\) and \(B\) in \(V(G)\) be disjoint sets.
The \emph{rank-connectivity} of $A$ and $B$ in $G$
is
\[\kappa^{\rk}_G(A,B)\coloneqq \min_{A\subseteq X\subseteq V(G)\setminus B}\rk_G(X,V(G)\setminus X).\]

\begin{lemma}[Monotonicity and subadditivity of the rank-connectivity function] \label{lem:properties_of_rk*}
If \(G\) is a graph and $A$, $B$, and $C$ are pairwise disjoint vertex subsets of $G$, then
\begin{enumerate}
\item \(\kappa^{\rk}_G(A,B)\leq \kappa^{\rk}_G(A,B\cup C)\), and
\item \(\kappa^{\rk}_G(A,B\cup C)\leq \kappa^{\rk}_{G-C}(A,B)+\kappa^{\rk}_{G-B}(A,C)\).
\end{enumerate}
\begin{proof}
The first inequality is immediate. For the second one, pick two cuts \(V(G)=A_B'\dotcup B'=A_C'\dotcup C'\)
	with \(A\subseteq A_B'\subseteq V(G)\setminus C\), \(A\subseteq A_C'\subseteq V(G)\setminus B\), \(B\subseteq B'\subseteq V(G)\setminus C\) and \(C\subseteq C'\subseteq V(G)\setminus B\) such that \(\rk(A_B',B')=\kappa^{\rk}_{G-C}(A,B)\) and \(\rk(A_C',C')=\kappa^{\rk}_{G-B}(A,C)\).
	Then
	\begin{align*}
	\kappa^{\rk}_G(A,B\cup C)
	&\leq \rk(A_B'\cap A_C',B'\cup C')\\
	&\leq \rk(A_B',B')+\rk(A_C',C')\\
	&\leq \kappa^{\rk}_{G-C}(A,B)+\kappa^{\rk}_{G-B}(A,C).
		\qedhere
	\end{align*}
\end{proof}
\end{lemma}

A \emph{rank-decomposition} of \(G\) is a rooted binary tree \(T\) together with a bijection between the leaves of this tree and the vertices of \(G\).
This way, every tree edge \(e\in E(T)\) induces a cut of \(V(G)\) into two parts \(A\) and \(B\) by partitioning the vertices according
to the component of \(T-e\) their leafs lie in. The width of the rank decomposition is the maximal rank \(\rk(A,B)\) across all these cuts,
and the rank-width of \(G\) is the minimal width of a rank decomposition.

The notable characterization of bounded rank-width we will need is the existence of well-linked sets of vertices.

A \emph{\(\rk\)-well-linked} set in a graph is a subset \(U\subseteq V(G)\) such that for any two disjoint sets
\(A,B\subseteq U\), we have \(\kappa^{\rk}_G(A,B)\geq \min\{|A|,|B|\}\).
\begin{theorem}[{\cite[Theorem 5.2]{rw_linked_sets}}]\label{thm:prelims:rw_well-linked_sets}
Every graph of rank-width greater than \(k\) contains a \(\rk\)-well-linked set of size \(k\).
\end{theorem}

There exist similar notions of well-linkedness for other connectivity parameters.
For example, a vertex subset \(U\subseteq V(G)\) is \emph{\(\mm\)-well-linked} if each cut \(U=X\dotcup (U\setminus X)\)
contains a matching of size at least \(\min\{|X|,|U\setminus X|\}\). This is equivalent to requiring that every two subsets
\(A,B\subseteq U\) are linked by \(\min\{|A|,|B|\}\) vertex-disjoint paths.
\begin{theorem}[{\cite[Lemma 3.4]{well-linked_sets}}]
Every graph of tree-width greater than \(k\) contains a \(\mm\)-well-linked set of size \(k\).
\end{theorem}

\paragraph*{Weisfeiler-Leman algorithms}
For a graph $G$ the $k${\nobreakdash-}dimensional Weisfeiler-Leman  algorithm (short: $k$-WL)
iteratively refines an isomorphism-invariant coloring
of $V(G)^k$.
The initial coloring
is $\chi^0_{G,k}: \vec{v} \mapsto \atp_G(\vec{v})$.
For $i \geq 0$ we set
\[\chi^{i+1}_{G,k}(\vec{v}) \coloneqq \Bigl(
\chi^i_{G,k}(\vec{v}),\multil[\bigl] \left(\atp_G(\vec{v}w),
\chi^i_{G,k}(\vec{v}\,[w/1]),\ldots,\chi^i_{G,k}(\vec{v}\,[w/k])\right)
\colon w \in V(G) \multir[\bigr]\Bigr),\]
where the notation \(\vec{v}\,[w/i]\) denotes the tuple obtained from \(\vec{v}\)
by replacing its \(i\)-th entry with \(w\).
Note that for some $i \in \N$,
the colorings $\chi^{i}_{G,k}$ and $\chi^{i+1}_{G,k}$ are equivalent, i.e., both colorings have precisely the same color classes.
If \(i\) is the minimal index where this happens, the \emph{stable coloring}
\(\chi_{G,k}^{i+1}\) is denoted by $\chi_{G,k}$ and
can be computed in $\mathcal{O}(k^2n^{k+1}\log n)$ \cite{wlruntime}.
Two graphs $G$ and $H$ are \emph{distinguished by}
$k${\nobreakdash-}WL
if there is a color $c$ such that
\[|\{ \vec{v} \in V(G)^k \colon \chi_{G,k}(\vec{v}) = c \}|
\neq |\{ \vec{w} \in V(H)^k \colon \chi_{H,k}(\vec{w}) = c \}|. \]
A graph $G$ is  \emph{identified} by $k${\nobreakdash-}WL
if the algorithm distinguishes $G$ from every non-isomorphic graph $H$.
The \emph{WL-dimension} of $G$ is the smallest
$k$ such that $k${\nobreakdash-}WL identifies the graph.
The WL-dimension of a graph class $\mathcal{C}$
is the smallest $k$ such that $k${\nobreakdash-}WL identifies
all graphs in~$\mathcal{C}$.
For a more detailed discussion of the algorithm, see for example \cite{cfi, wl_kiefer}.

The logic $C_k$ is the $k$-variable fragment of first-order logic with counting.
For two graphs~$G$ and~$H$ that satisfy exactly the same formulas of $C_k$,
we write $G \equiv_{C_{k}} H$.
For a more detailed definition and relevant examples,
see for example \cite{cfi}.
\begin{theorem}[Theorem~5.2 of~\cite{cfi}]
        Let $k \in \mathbb{N}$. Two graphs $G$ and $H$ are not distinguished by
        		$k${\nobreakdash-}WL precisely if
        		 $G \equiv_{C_{k+1}} H$.
\end{theorem}

The \emph{bijective $k$-pebble game},
denoted by $\bpk$,
is a round-based combinatorial game for two players, \emph{Spoiler} and \emph{Duplicator}.
A \emph{position}  is a pair of tuples $(\vec{v},\vec{w})
\in V(G)^\ell \times V(H)^\ell$ with $\ell \in [k] \cup \{0\}$.
The initial position is $((),())$.
Let $((v_1,\ldots,v_\ell),(w_1,\ldots,w_\ell))$  be the position at the start of a round.
If ${\ell > 0}$, then Spoiler can remove one pebble pair by picking some
$i \in [\ell]$ and move to
$((v_1,\ldots,v_{i-1},v_{i+1},\ldots,v_\ell,)(w_1,\ldots,w_{i-1},w_{i+1},\ldots,w_\ell))
\eqqcolon (\vec{v}',\vec{w}')$.
If Spoiler does not remove a pebble pair, then
$(\vec{v}',\vec{w}') \coloneqq (\vec{v},\vec{w})$.
If Spoiler removed a pebble or
$\ell < k$, then Duplicator picks a bijection $b: V(G) \rightarrow V(H)$ and Spoiler picks some $v \in V(G)$ and moves to $(\vec{v}'v,\vec{w}'b(v))$.
Spoiler wins if at any point,
$\atp_G(\vec{v}) \neq \atp_H(\vec{w})$.
Duplicator wins if Spoiler never wins.
    We say $p$ is a \emph{winning position} for Spoiler/Duplicator
if Spoiler/Duplicator has a winning strategy for the game starting
at~$p$.
We say Spoiler/Duplicator \emph{wins}
if $((),())$ is a winning position for Spoiler/Duplicator.
\begin{lemma}[\cite{cfi}]
        For all integers $k \geq 1$, two graphs $G$ and $H$
        are distinguished by $k${\nobreakdash-}WL if and only if
        Spoiler wins the bijective ${(k+1)}${\nobreakdash-}pebble game on $G$ and $H$.
\end{lemma}

\paragraph*{Half-graphs and stability}
For an integer \(k\), the \emph{half-graph} \(H_k\) is the bipartite graph with vertex set
\(V(H_k)=\{v_i\colon i\in[k]\}\cup\{w_i\colon i\in[k]\} \) such that
\(v_iw_j\in E(H_k)\) if and only if \(i\leq j\).

A class of graphs is monadically stable if it does not transduce the class of all half-graphs.
In particular, this implies that the class is \(H_k\)-semi-free for some \(k\in\N\).
Since we will only use this latter property of monadically stable classes, we will not
formally define \(\FO\)-transductions, and instead refer to \cite{stable_bounded_tww}.

\section{Twin-width 4}

In this section, we show that 
the graphs of twin-width \(4\) do not
have a bounded Weisfeiler-Leman dimension.
To this end, we consider subdivisions of
\emph{Cai-Fürer-Immerman} graphs (short: CFI graphs), cf.~\cite{cfi}.

\begin{lemma}[\cite{cfi}]\label{cfi}
    For all $k \in \N$ with $k \geq 1$,
    there is a pair of non-isomorphic cubic CFI graphs
    of size~$\mathcal{O}(k)$
    which are not distinguished by
    $k$-WL.
\end{lemma}

    Let $G$ be a graph and $s \in \N$.
    The \emph{$s$-subdivision} of $G$, denoted by $G^s$ is the graph obtained from $G$ by replacing
   every edge of $G$ by a path of length $s+1$.
    A \emph{$(\geq s)$-subdivision} of $G$ is an $s'$-subdivision of $G$
    with $s' \geq s$.

\begin{lemma}[\cite{twwleq4NPharda}]\label{tww4-subdivisions}
	If $G$ is a graph on $n$ vertices
	and $s \geq 2 \log n$, then $\tww(G^s) \leq 4$.
\end{lemma}

Since two graphs are isomorphic precisely if their subdivisions are isomorphic,
it follows that the general graph isomorphism problem
can be reduced to the graph isomorphism problem on the
 class of graphs of twin-width at most \(4\).
In particular, if the WL-dimension of twin-width \(4\) graphs were bounded, then graph isomorphism would be polynomial-time solvable.
In this section, we prove the dimension to be unbounded.

\begin{lemma} \label{tww4-logics}
    If $G$ and $H$ are graphs with $G \equiv_{C_{2k}} H$, then $G^1 \equiv_{C_k} H^1$.
\end{lemma}

\begin{proof_sketch}
	1-subdivisions can be defined via 2-dimensional logical interpretations.
	A detailed proof as well as preliminaries on interpretations are in
	Appendix~\labelcref{appendix:interpretations}.
\end{proof_sketch}

This result can be extended to $s$-subdivisions
for $s>1$.

\begin{theorem} \label{tww4-pebbles-restated}
Let $s \in \N$, and $G,H$ be graphs
with $V(G) = V(H)$.
Then for all $k$, if $2k$-WL does not distinguish $G^{3}$
and $H^{3}$,
then $k$-WL does not distinguish $G^s$ and $H^s$.
\end{theorem}

\begin{proof_sketch}
	A full proof for this can be found in Appendix~\labelcref{tww4:pebbleproof}.
	Intuitively, Duplicator extends its strategy for 3-subdivided graphs to $s$-subdivided
	graphs by copying the bijection
	on the original vertices,
	then looking ahead one move
	to see the “direction” that the path
	vertices should be mapped in in the original game,
	and adapting the bijection accordingly.
\end{proof_sketch}

By applying \cref{tww4-logics} twice
and then \cref{tww4-pebbles-restated},
we obtain the following result.

\begin{corollary}\label{pebblecorollary}
Let $s,k \in \mathbb{N}$ and let $G$ and $H$ be graphs.
If $G \cequiv{8k} H$, then $G^{s} \cequiv{k} H^{s}$.
\end{corollary}

Finally, this gives us our theorem.

\begin{theorem}
For all $k \in \N$, there are two non-isomorphic graphs of twin-width \(4\)
and size $\mathcal{O}(k \log k)$ that
are not distinguished by $k$-WL.
\end{theorem}

\begin{proof}
Take the two graphs $G_{8k},H_{8k}$ from \cref{cfi},
and denote their size by $n \coloneqq |V(G_{8k})|$.
Note that $n \in \mathcal{O}(k)$.
Let $s \coloneqq 2 \lceil \log n \rceil$
and consider $G_{8k}^s$ and $H_{8k}^s$.
Since $G_{8k}$ and $H_{8k}$ are cubic,
their number of edges is in $\mathcal{O}(k)$.
In total, the size of their subdivisions is in
$\mathcal{O}(k+ks) = \mathcal{O}(k \log k)$.
By \cref{tww4-subdivisions},
both of them have twin-width at most \(4\),
and since $G_{8k}$ and $H_{8k}$ are not isomorphic,
their subdivisions are also not isomorphic.
Finally, we know by \cref{pebblecorollary} that Duplicator wins
the $k$-pebble game on
the subdivisions,
and therefore $k$-WL does not distinguish the two graphs.
\end{proof}
\newcommand{\cs}{\operatorname{cs}}
\newcommand{\CS}{\operatorname{CS}}
\newcommand{\itp}{\operatorname{itp}}
\newcommand{\ntw}{\operatorname{ntw}}
\newcommand{\ModTree}{\mathrm{ModTree}}

\section{Twin-width 1}

We now know that the class of graphs of twin-width 4 does not have bounded
WL-dimension.
The pressing question is whether the dimension of classes of smaller twin-width is bounded.
The class of twin-width 0 graphs,
i.e., the class of cographs, has WL-dimension 2, see~\cite{2wl-cographs}.
The following is known about twin-width 1 graphs.

\begin{lemma}[\cite{tww_polynomial_kernels}]\label{lem:tww1-sequence}
Graphs of twin-width 1 have a 1-sequence
where there is a single red edge at every step.
\end{lemma}

\begin{lemma}[\cite{rank-width}]\label{lem:WL-dim_vs_rw}
	For every $k \in \N$,
	the (3$k$+4)-dimensional Weisfeiler-Leman algorithm identifies
	every graph of rank width at most $k$.
\end{lemma}

\begin{corollary}
Graphs of twin-width 1 have WL-dimension at most 10.
\end{corollary}
\begin{proof}
By Lemma~\ref{lem:tww1-sequence}, graphs of twin-width 1 have component twin-width 2 and, hence, by Theorem 10 of~\cite{tww6a},
also bounded boolean width
and bounded rank-width~\cite{boolean-width}.
By inspecting the proof of Theorem 10 of~\cite{tww6a},
one can see that graphs of component twin-width \(2\)
have rank-width at most \(2\):
As rank decomposition we consider the contraction tree of the red components
along the contraction sequence of component twin-width \(2\).
The claim then follows from \Cref{lem:WL-dim_vs_rw}.
\end{proof}

Building upon the result in \cite{tww_polynomial_kernels},
one can also show a better bound much more explicitly:
Namely, graphs of twin-width 1 have WL-dimension at most 3.
First, we show this for the case of prime graphs.

\subsection{Prime graphs of twin-width 1}
A vertex $v$ of a graph $G$ is \emph{homogeneously connected} to a set of vertices $S \subseteq V(G)$
if it is adjacent to all or none of $S$ in $G$.
A set $M \subseteq G$ is a \emph{module} if all vertices $v \in V(G) \setminus M$
are homogeneously connected to $M$.
A module of a graph $G$ is \emph{trivial} if it consists of a single vertex or
all vertices of~$G$.
A graph $G$ is \emph{prime} if it has only trivial modules.

In this section,
we further examine
how the proof of \cref{lem:tww1-sequence}
shows that any graph of twin-width 1
has a 1-sequence of the following form:
\begin{enumerate}
\item Guess the first contraction of a valid 1-contraction sequence.
This creates one red edge.
\item Make every possible \enquote{safe} contraction,
that is,
contractions into the ends of the current red edge that create no additional red edges.
\item Contract the ends of the red edge.
If there is only one vertex left, terminate. Otherwise, this will create
exactly one more red edge. Repeat Step~2.
\end{enumerate}

This contraction sequence is \emph{canonical}
in the sense that given the starting
contraction, there is a description
of it that is an invariant
for colored prime graphs of twin-width~1
such that if the invariant is the same on two graphs,
the graphs are isomorphic.
Given a colored prime graph $G^c$ of twin-width~1
and two vertices $u,v$ in it,
we define the function $\cs(G^c,u,v)$ which returns
a description of this sequence starting from contracting
$u$ and $v$ if it exists, and $\texttt{f}$ if it doesn't.
For input graphs with a single vertex,
return the color of that vertex. Otherwise:

\begin{enumerate}
\item Set $i \coloneqq 0$. If $uv \in E(G)$, set
$\mathrm{cs}_0 \coloneqq \texttt{1}c(u)c(v)$, otherwise set $\mathrm{cs}_0 \coloneqq \texttt{0}c(u)c(v)$.
\item Set $G_0 \coloneqq (V(G),E(G),\emptyset)$, set
$u_{G,0} \coloneqq u$, and $v_{G,0} \coloneqq v$.
Next, iterate the following two steps:
\begin{enumerate}
\item If $|V(G_i)| = 2$, then
terminate with $\cs(G^c,u,v) \coloneqq \mathrm{cs}_i$.
Otherwise, let $W$ be the set of vertices in $V(G_i)$ adjacent to
exactly one of $u_{G,i}$ and $v_{G,i}$.
If $|W| \neq 1$, then terminate with $\cs(G^c,u,v) \coloneqq \texttt{f}$.
Otherwise, let $w \in W$,
and set $\mathrm{cs}_{i+1} \coloneqq \mathrm{cs}_{i}
\texttt{\#r0}c(w)$ if $wv_{G,i} \in E(G_i)$,
and $c_{i+1} \coloneqq \mathrm{cs}_{i}
\texttt{\#r1}c(w)$ if $wu_{G,i} \in E(G_i)$.
Set $G_{i+1} \coloneqq G_i/u_{G,i},v_{G,i}$,
set $u_{G,i+1}$ to be the vertex resulting from the contraction,
and $v_{G,i+1} \coloneqq w$. Increase $i$ by one and go to step 2b.
\item
For this step, we define \emph{near-twins}
with a function $\ntw_{G}^{(t,t'),x,c^*}(u,v)$ where $G$ is a trigraph
with some coloring $c$,
$u$ and $v$ are vertices in $G$,
$t$ and $t'$ are atomic types, $x \in \{u,v\}$, and~$c^*$ is a color.
Then the function is defined via
\begin{align*}
	\ntw_G^{(t,t'),x,c^*}(u,v) \coloneqq \{w \in V(G) \mid \,
	&\atp_G(u,w) = t,\atp_G(v,w) = t', \\
	&G/w,x = G \setminus \{w\}, \\
	&c(w) = c^* \}.
\end{align*}
Further, define $t_1$ as the atomic type for an edge
and $t_0$ the atomic type for a non-edge.
Let $k$ be the number of colors used by $c$.
Set $n_i^{(t,t'),x,c} \coloneqq |\ntw_{G_i}^{(t,t'),x,c}(u_{G_i},v_{G_i})|$, and
$\mathcal{T} \coloneqq ((t_0,t_0),(t_0,t_1),(t_1,t_0),(t_1,t_1))$.
We set
\begin{align*}
n_i^x &\coloneqq
n_i^{\mathcal{T}_1,x,1};\dots;n_i^{\mathcal{T}_4,x,1};
\dots;n_i^{\mathcal{T}_4,x,k},\\
\mathrm{cs}_{i+1} &\coloneqq \mathrm{cs}_i
\texttt{\#}n_i^{u_{G,i}}\texttt{;}n_i^{v_{G,i}}.
\end{align*}
Let $G_{i+1}$ be the trigraph obtained from $G_i$ by 
performing all the safe contractions as described above.
From here on, each of this series of contractions is referred to
as a \emph{contraction phase}.
Set $u_{G,i+1} \coloneqq u_{G,i}$, and $v_{G,i+1} \coloneqq v_{G,i}$.
Increase $i$ by one and go to the next iteration.
\end{enumerate}
\end{enumerate}
\noindent
In the $i$-th iteration,
the red edge in the trigraph $G_i$ is always between $u_{G,i}$ and $v_{G,i}$,
and as such the sequence of trigraphs $G_i$ and vertices $u_{G,i}$ and $v_{G,i}$
represents the canonical contraction sequence starting at $u$ and $v$.
This function is polynomial-time computable
and the runtime is $\mathcal{O}(|V|^3)$.
We can show:

\begin{lemma}\label{csguv}
Let $G^c=(G,c_G)$, and $H^c=(H,c_H)$
be two colored prime graphs of twin-width~1.
The two graphs are isomorphic if and only if
for some $u,v \in V(G)$ and $u',v' \in V(H)$
it holds that
$\cs(G^c,u,v) = \cs(H^c,u',v') \neq \texttt{f}$.
\end{lemma}
\begin{proof_sketch}
Because the canonical contraction sequence
defines an order on the vertices of a graph,
it is clear how to derive an isomorphism
between two graphs with the same canonical contraction sequence.
The full proof can be found in Appendix \labelcref{cssequence}.
\end{proof_sketch}

\begin{theorem}\label{tww1-prime-invariant}
There is a polynomial-time computable
invariant $I$ for colored prime graphs of twin-width 1
such that for any graphs $G,H$,
if $I(G) = I(H)$ then $G \cong H$.
\end{theorem}
\begin{proof}
Let $G$ be a colored prime graph of twin-width 1.
Define
\[\CS(G) \coloneqq \{ \cs(G,u,v) \mid u,v, \in V(G), \cs(G,u,v) \neq \texttt{f} \}.\]
Then $\cs(G) \coloneqq \min_{\leq_\text{lex}}(\CS(G))$ is such an
invariant by \cref{csguv} and computable in time $\mathcal{O}(|V|^5)$,
because the canonical contraction sequence is computable in time $\mathcal{O}(|V|^3)$.
\end{proof}

By \enquote{unwinding}
the minimal canonical contraction sequence
for colored prime graphs of twin-width 1,
one can also get a canonical form.

Finally, we can leverage
the canonical contraction sequence
to show that 3-dimensional Weisfeiler-Leman
identifies colored prime graphs of twin-width 1.

\begin{theorem}\label{3wl-coloredprimetww1}
Let $G,H$ be two colored prime graphs of twin-width 1 such that $G \not \cong H$.
Then 3-WL distinguishes $G$ and $H$.
\end{theorem}

\label{tww1:cssequence-pebble}

\begin{proof}
We show that Spoiler wins
the 4-pebble game on $G$ and $H$.
The idea of the strategy is to fix the starting contraction pair
with a pebble each,
then track the resulting sequence with two more pebbles
and eventually spot the difference.

Let $u,v \in V(G)$ be some pair of vertices such that $\text{cs}(G,u,v) \neq \text{f}$.
Spoiler lays a pebble on $u$ and then a pebble on $v$. 
Let $b$ be the bijection Duplicator responds with after both are pebbled,
and set $u' \coloneqq b(u)$ and $v' \coloneqq b(v)$.
Observe that $\text{cs}(G,u,v) \neq \cs(H,u',v')$, because otherwise
the two graphs would be isomorphic by \cref{csguv}.

From here, Spoiler never moves the first two pebbles,
and Duplicator must try to replicate
the contraction sequence described in $\cs(G,u,v)$,
or else Spoiler can use the third and fourth pebble to spot the difference and win.
However, since $\cs(G,u,v) \neq \cs(H,u',v')$,
Spoiler will always be able to spot a difference at some point
in the sequence.
To do this, we define a series of sets $U_{G,i}$ and $V_{G,i}$
and a series of sets $U_{H,i}$ and $V_{H,i}$
such that the following invariant holds in all rounds:
\begin{enumerate}
	\item If Duplicator does not choose a bijection $b$
	such that $b(U_{G,i}) = U_{H,i}$ and $b(V_{G,i}) = V_{H,i}$,
	then Spoiler has a winning strategy.
	\item $U_{G,i}$, $V_{G,i}$, $U_{H,i}$ and $V_{H,i}$ are modules.
	\item The set $U_{G,i}$ always contains a pebbled vertex.
\end{enumerate}
The start of the series is
$U_{G,i} \coloneqq \{u\},V_{G,i} \coloneqq \{v\},
U_{H,i} \coloneqq \{u'\},U_{G,i} \coloneqq \{v'\}$.
Spoiler will never move the first two pebbles until the very
end of the game, so the invariant trivially holds for these sets.
From here, Spoiler plays as follows:
\begin{itemize}
\item For the very first contraction indicated in $\cs(G,u,v)$,
Spoiler automatically wins if $\atp_G(u,v) \neq \atp_H(u',v')$
or the colors differ.
Otherwise, continue.
\item If the next contraction described in $\cs(G,u,v)$ contracts a
red edge or is the very first contraction,
there must be exactly one $w \in V(G)$ that is connected to all of exactly one of
$U_{G,i},V_{G,i}$ and none of the other. There must also be exactly one
$w' \in V(H)$ of the same color that is connected to all of exactly one of
$U_{H,i},V_{H,i}$.
If Duplicator does not map $w$ to $w'$
or $w$ and $w'$ differ in their
color or isomorphism type regarding the respective sets,
Spoiler pebbles~$w$.
Now, if Duplicator maps $U_{G,i}$ to $U_{H,i}$
and $V_{G,i}$ to $V_{H,i}$, Spoiler wins:
Let $b$ be the bijection chosen by Duplicator.
There is always a pebble in $U_{G,i}$ on some vertex $u^* \in U_{G,i}$
and $w$ is homogeneously connected
to $U_{G,i}$ and $V_{G,i}$, and so for any $v^* \in V_{G,i}$,
it holds that $\atp_G(u^*,w,v^*) \neq \atp_H(b(u^*),b(w),b(v^*))$.
If Duplicator does not map $U_{G,i}$ to $U_{H,i}$
and $V_{G,i}$ to $V_{H,i}$, Spoiler \enquote{backtracks}
to those sets and follows the winning strategy which
must exist due to the invariant for those two sets.
If Duplicator forces Spoiler to backtrack continuously until
the start of the contraction sequence, Spoiler
wins due to $u$ and $v$ being pebbled.
Finally, if Duplicator maps $w$ to $w'$
and they agree in their color and isomorphism types
regarding the respective $U$ and~$V$,
continue with $U_{G,i+1} \coloneqq U_{G,i} \cup V_{G,i}$,
$U_{H,i+1} \coloneqq U_{H,i} \cup V_{H,i}$,
$V_{G,i+1} \coloneqq \{w\}$ and $V_{H,i+1} \coloneqq \{w'\}$.
See \Cref{fig:redcontraction}.

	\begin{figure}
			\centering
			\begin{tikzpicture}[minimum size = 1em,
    				baseline=(current bounding box.north),
    				every path/.style={thick}]
    			\node (G) at (-2,3) {$G:$};
   			\node (ugi) at (0,2) {$U_{G,i}$};
			\begin{scope}[every node/.style={circle,draw}]
   			\node (u) at (0,0) {$u$};
   			\node (v) at (0,1) {$v$};
   			\node (w) at (2,-0.2) {$w$};
   			\node (vgi) at (0,-2) {$V_{G,i}$};
   			\node[ellipse,draw=black,inner sep=0,
    				fit = (ugi) (v) (u)] (ugibub) {};
   			\node[label=left:{$U_{G,i+1}$},ellipse,dashed,draw=black,inner sep=0,
    				fit = (ugibub) (vgi)] (ugiplus1) {};
   			\node[label={[label distance=-3mm]below:$V_{G,i+1}$},dashed,draw=black,inner sep=0,
    				fit = (w)] (vgiplus1) {};
			\end{scope}
    			\node (H) at (4,3) {$H:$};
   			\node (uhi) at (7,2) {$U_{H,i}$};
			\begin{scope}[every node/.style={circle,draw}]
   			\node (up) at (7,0) {$u'$};
   			\node (vp) at (7,1) {$v'$};
   			\node (wp) at (5,-0.2) {$w'$};
   			\node (vhi) at (7,-2) {$V_{H,i}$};
   			\node[ellipse,draw=black,inner sep=0,
    				fit = (uhi) (vp) (up)] (uhibub) {};
   			\node[label=right:{$U_{H,i+1}$},ellipse,dashed,draw=black,inner sep=0,
    				fit = (uhibub) (vhi)] (uhiplus1) {};
   			\node[label={[label distance=-3mm]below:$V_{H,i+1}$},dashed,draw=black,inner sep=0,
    				fit = (wp)] (vhiplus1) {};
			\end{scope}
    			\draw[draw=red] (ugibub) -- (vgi);
    			\draw[draw=red] (uhibub) -- (vhi);
			\path[->, draw=blue] (w) edge (wp);
			\path[->, draw=blue] (ugibub) edge (uhibub);
			\path[->, draw=blue] (vgi) edge (vhi);
		\end{tikzpicture}
	    	\caption{Illustration of a red edge contraction in the proof for
	    	\Cref{3wl-coloredprimetww1}.
    		The bijection chosen by Duplicator is indicated by the arrows.}
    		\label{fig:redcontraction}
    	\end{figure}
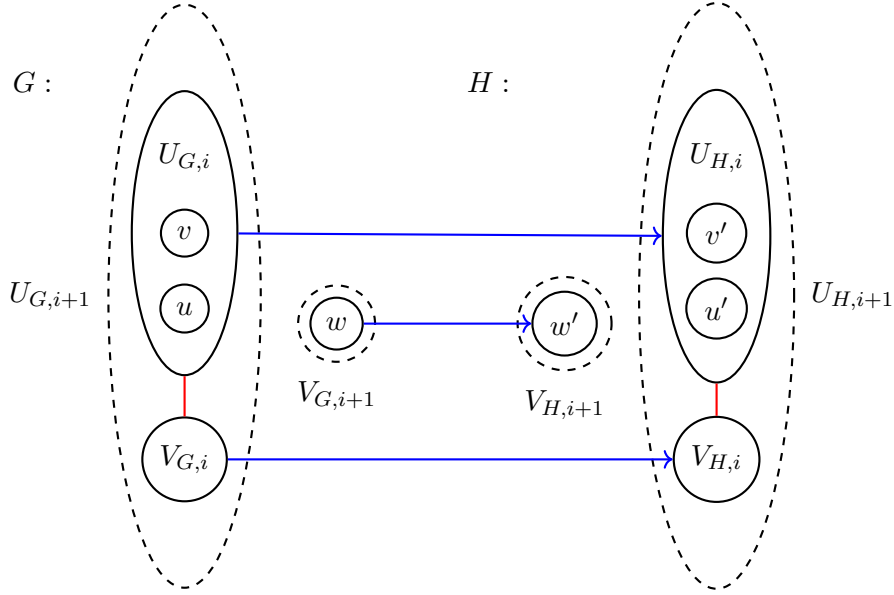
    	
Then the invariant holds:
\begin{enumerate}
\item The first part of the invariant holds as
Duplicator must always map the set $U_{G,i+1}$ to $U_{H,i+1}$
because otherwise she either did not map $U_{G,i}$ to $U_{H,i}$
or she did not map $V_{G,i}$ to $V_{H,i}$, and the invariant holds for those sets.
Duplicator must also always map $w$ to $w'$ as otherwise
Spoiler wins by pebbling $w$ as described above.
\item The second part of the invariant holds
as the sets follow the canonical contraction sequence:
The vertex represented by $U_{G,i}$ and the vertex represented by $V_{G,i}$
get contracted in this step and form $U_{G,i+1}$,
which creates a red edge to $w$, now represented by $V_{G,i+1}$,
but no other vertex. Thus, they agree on all other vertices.
Same for $H$.
\item
Observe that the vertices ${u,v \in U_{G,i} \subseteq U_{G,i+1}}$ are always pebbled.
The third part of the invariant
also holds.

\end{enumerate}
\item For all other contraction phases described in $\cs(G,u,v)$,
let $W$ be the set of vertices chosen for contraction in $G$,
and let $W'$ be the set of vertices chosen for contraction in $H$.
If $|W| \neq |W'|$ or Duplicator does not map $W$ to $W'$,
let~$b$ be the bijection chosen by Duplicator
and without loss of generality let $w \in W$
such that $w' \coloneqq b(w) \not \in W'$.
Spoiler now pebbles $w$.
If Duplicator maps the set $U_{G,i}$ to $U_{H,i}$
and $V_{G,i}$ to $V_{H,i}$, Spoiler wins:
Either $w'$ does not have the right color or
isomorphism type regarding the corresponding
$U$ and~$V$. 
In this case, there is always a pebble in $U_{G,i}$ on some vertex $u^* \in U_{G,i}$
and~$w$ is homogeneously connected
to $U_{G,i}$ and $V_{G,i}$, and so for any $v^* \in V_{G,i}$,
it holds that
$\atp_G(u^*,w,v^*) \neq \atp_H(b(u^*),b(w),b(v^*))$.
Otherwise, $w'$ disagrees on a common neighbor with the respective
$U$ or $V$.
Then Spoiler can win by moving
the third pebble anywhere into $V_{G,i}$,
and then either backtracking if Duplicator triggers the invariant
for $U_{G,i}$ and $V_{G,i}$,
or finally pebbling the neighbor that is disagreed on.
If $|W| = |W'|$
and Duplicator does map
$W$ to~$W'$, continue by adding $W$ to $U_{G,i}$ to get $U_{G,i+1}$
and $W'$ to $U_{H,i}$ to get $U_{H,i+1}$
or by adding $W$ to $V_{G,i}$ to get $V_{G,i+1}$
and $W'$ to $V_{H,i}$ to get $V_{H,i+1}$,
depending on the contraction phase, leaving the other set unchanged
for the next iteration. See \Cref{fig:contractionphase}.

\begin{figure}
			\centering
			\begin{tikzpicture}[minimum size = 1em,
    				baseline=(current bounding box.north),
    				every path/.style={thick}]
    			\node (G) at (-2,3) {$G:$};
   			\node (ugi) at (0,2) {$U_{G,i}$};
			\begin{scope}[every node/.style={circle,draw}]
   			\node (u) at (0,0) {$u$};
   			\node (v) at (0,1) {$v$};
   			\node (w) at (2,-0.2) {$W$};
   			\node (vgi) at (0,-3) {$V_{G,i}$};
   			\node[ellipse,draw=black,inner sep=0,
    				fit = (ugi) (v) (u)] (ugibub) {};
   			\node[label=left:{$U_{G,i+1}$},ellipse,dashed,draw=black,inner sep=0,
    				fit = (ugibub) (w)] (ugiplus1) {};
   			\node[label={[label distance=-3mm]below:$V_{G,i+1}$},dashed,draw=black,inner sep=0,
    				fit = (vgi)] (vgiplus1) {};
			\end{scope}
    			\node (H) at (4,3) {$H:$};
   			\node (uhi) at (7.5,2) {$U_{H,i}$};
			\begin{scope}[every node/.style={circle,draw}]
   			\node (up) at (7.5,0) {$u'$};
   			\node (vp) at (7.5,1) {$v'$};
   			\node (wp) at (5.5,-0.2) {$W'$};
   			\node (vhi) at (7.5,-3) {$V_{H,i}$};
   			\node[ellipse,draw=black,inner sep=0,
    				fit = (uhi) (vp) (up)] (uhibub) {};
   			\node[label=right:{$U_{H,i+1}$},ellipse,dashed,draw=black,inner sep=0,
    				fit = (uhibub) (wp)] (uhiplus1) {};
   			\node[label={[label distance=-3mm]below:$V_{H,i+1}$},dashed,draw=black,inner sep=0,
    				fit = (vhi)] (vhiplus1) {};
			\end{scope}
    			\draw[draw=red] (ugibub) -- (vgi);
    			\draw[draw=red] (uhibub) -- (vhi);
			\path[->, draw=blue] (w) edge (wp);
			\path[->, draw=blue] (ugibub) edge (uhibub);
			\path[->, draw=blue] (vgi) edge (vhi);
		\end{tikzpicture}
	    	\caption{Illustration of contraction phase in the proof for
	    	\Cref{3wl-coloredprimetww1}.}
    		\label{fig:contractionphase}
    	\end{figure}
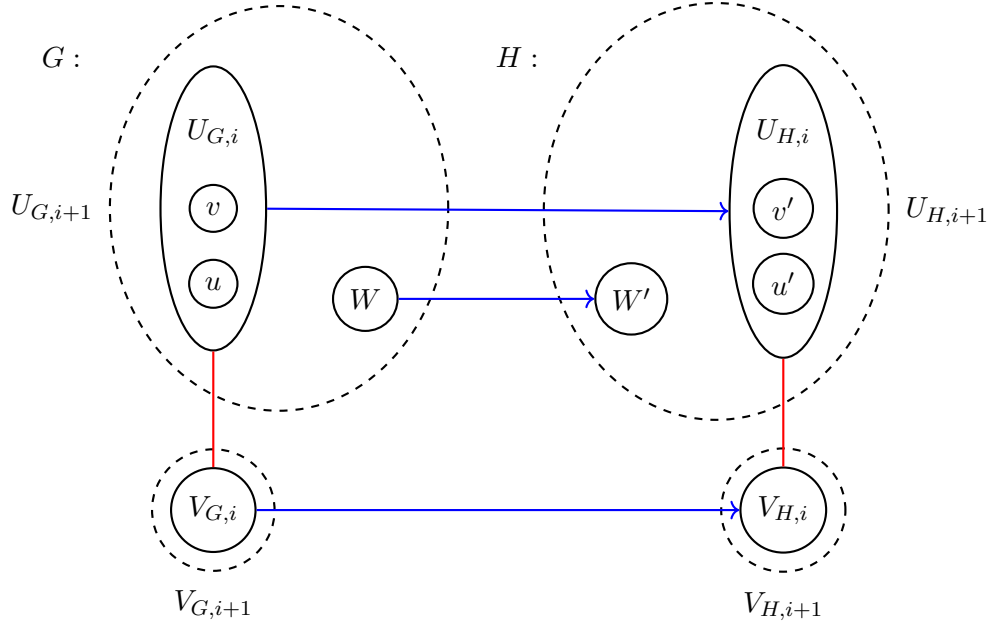
    	
Then the invariant holds:
\begin{enumerate}
\item The first part of the invariant holds,
as Duplicator must always map $W$ to $W'$ as otherwise
Spoiler wins as described above.
\item The second part of the invariant
holds as the sets follow the canonical contraction sequence:
The vertex represented by $U_{G,i}$ and the vertices in $W$
get contracted and are now represented by $U_{G,i+1}$,
and $V_{G,i+1}$ stays the same as $V_{G,i}$.
Same for $H$.
\item The third part holds,
since the vertices $u,v \in U_{G,i} \subseteq U_{G,i+1}$ are always pebbled.

\end{enumerate}
\end{itemize}
This process will eventually result in a victory for Spoiler
as the two canonical contraction sequences must differ at some point.
In essence, Spoiler finds the earliest mistake in the contraction
sequence emulated by Duplicator, and then at worst
backtracks to the start of the sequence to \enquote{trap} Duplicator
and win. Thus, Spoiler wins after at most $|V(G)|$ moves.
\end{proof}

\subsection{From prime graphs to all graphs}

Now, we lift the result to all graphs of twin-width 1.
For this, we use \textit{modular decompositions}.
Recall the definition
of modules that is at the core of prime graphs.
A set of modules is \emph{strong} if the modules in it are all pairwise disjoint.
For a connected and coconnected graph $G$,
let $M_G$ be the inclusion-wise maximal strong set of modules
that are proper subsets of $V(G)$. The elements of $M_G$ will also
be referred to as \emph{maximal modules} of $G$.

The set $M_G$ is a uniquely determined partition of $V(G)$
\cite{strongmodules}.
We will call $M_G$ the modular decomposition of $G$.
Gallai has shown a nice way to continually decompose a graph
into its maximal modules, where they always yield a prime quotient graph:

\begin{lemma}[\cite{mod-decomp-Gallai}]
For every graph $G$, one of the following holds:
\begin{itemize}
\item $G$ is disconnected.
\item $\overline{G}$ is disconnected.
\item $G/M_G$ is prime.
\end{itemize}
\end{lemma}

Note that because the sets of $M_G$ are modules of $G$,
$G/M_G$ will have no red edges.
Inspired by Gallai's Theorem,
there is a unique tree representing all modules of a graph.

\begin{definition}[Modular decomposition tree, \cite{modtrees}]
For a graph $G$,
\emph{the modular decomposition tree} $\ModTree_G = (\mathcal{M},E,\ell)$
consists of a set of modules $\mathcal{M} \subseteq 2^V$ of $G$,
a relation $E \subseteq \mathcal{M}^2$ such that $(\mathcal{M},E)$ is a directed tree,
and a label function $\ell: \mathcal{M} \rightarrow
\{\texttt{prime}, \texttt{series}, \texttt{parallel}, \texttt{single}\}$
labeling each module.
It is defined recursively as follows:
\begin{itemize}
\item The modular decomposition tree of a graph with one vertex $v$
is the module $\{v\}$ labeled \texttt{single}.
\item If $G$ is connected and coconnected, then label the root
$V(G)$ \texttt{prime},
and for each $M \in M_G$,
attach $\ModTree_{G[M]}$ to it as a subtree.
\item If $G$ is disconnected, then label the root
$V(G)$ \texttt{parallel},
and for each connected component $C$ of $G$,
attach $\ModTree_{G[C]}$ to it as a subtree.
\item If $\overline{G}$ is disconnected, then label the vertex
$V(G)$ \texttt{series},
and for each connected component $C$ of $\overline{G}$,
attach $\ModTree_{G[C]}$ to it as a subtree.
\end{itemize}
\end{definition}

Using this, we can lift our canonization
results from prime graphs to all graphs:
By lexicographically ordering the subgraphs in each step,
we can recursively define a canonization
for the whole graph by using the canonization
of the relevant subgraphs.
In order to do this,
we can use the canonization of \textit{colored} prime graphs,
coloring each vertex in the quotient prime graph
with the isomorphism type of the module it represents.

\begin{lemma}\label{tww1-canonicalform}
There is a polynomial-time computable canonical form
for colored graphs of twin-width 1.
\end{lemma}

Next, we turn to distinguishing
graphs of twin-width 1 with Weisfeiler-Leman.
First, we show that 2-WL distinguishes maximal modules in a graph.

\begin{lemma}\label{2wl-modules}
Let $G,H$ be connected and coconnected graphs.
For all $u,v \in V(G)$ and $u',v'\in V(H)$,
if $u$ and $v$ are in the same maximal module of $G$
and $u'$ and $v'$ are in different maximal modules of $H$,
then $uv,u'v'$ is a winning position for Spoiler
in $\bp{3}(G,H)$.
\end{lemma}

\begin{proof}
One can calculate the modular decomposition
tree of a graph $G$ bottom-up, assigning levels for each module:
Start by tracing back all possible series and parallel decompositions,
then a prime step, then repeat until the root of the tree is formed.
Given this labeling of the tree,
one can show inductively that on all levels starting from the leaves,
the Lemma is true.

For a graph $G$,
we define $\text{twins}(G)$ to be the partition of $V(G)$
into equivalence classes of the \enquote{is-twin} relation on $G$,
that is for all $u,v \in V(G)$,
they are in the same set of $\text{twins}(G)$
if and only if they are twins.
The modular decomposition tree of a graph $G$ can be calculated bottom-up as follows:
Set $M_0 \coloneqq \{ \{v\} \mid v \in V(G) \}$. Repeat while $M_i \neq \{V(G)\}$:
\begin{enumerate}
\item Let $T \coloneqq \text{twins}(G/M_i)$.
If $T$ is the trivial partition into singletons, skip to step b.
Otherwise,
set $M_{i+1} \coloneqq \{ \bigcup_{M \in t} M \mid t \in T \}$.
In this way, the vertices in each twin-class of $G/M_i$ are 'merged' together.
Note that each twin-class $t$ either forms an independent set or a clique
in $G/M_i$.
In the case that $t$ forms an independent set, this 'merge' corresponds
to a parallel decomposition,
otherwise a series decomposition.
Repeat this step with $M_{i+1}$.
\item Let $C$ be the partition of $V(G / M_i)$ into connected components,
and $C'$ its further refinement into coconnected components of the connected
components.
Set $M_{i+1} \coloneqq \{ M_{(G/M_i)[c]} \mid c \in C' \}$,
that is, the maximal modules of each component in $C'$.
This corresponds to a prime decomposition:
The vertices of each prime part of $G/M_i$ get 'merged'.
Loop back to the first step and continue with $M_{i+1}$.
\end{enumerate}
Following this procedure,
a modular decomposition tree of $G$ is defined with assigned \enquote{levels}
to each module, and in fact since $\text{ModTree}_G$ is unique,
exactly that composition is defined.
We say that for a graph $G$ and an $i \in \N$, a module $M \in M_i$
is a level $i$ module of $G$.
Now, we show by induction on $i$ that for all $u,v \in V(G)$
and $u',v' \in V(H)$, if $u$ and $v$ are in the same level $i$ module of $G$,
and $u'$ and $v'$ are not in the same level $i$ module of $H$,
then Spoiler wins $\bp{3}(G,H)$ from position $uv,u'v'$.
The maximal modules are those at the top level,
so the lemma follows.
\begin{itemize}
\item For level 0, the statement holds trivially.
\item Let $u,v$ be two vertices in the same level $i+1$ module of $G$.
Let $u',v' \in V(H)$ such that $u'$ and $v'$ are in different level $i$ modules
of~$H$.
If $u$ and $v$ are also in the same level $i$ module of $G$,
the statement holds by induction
(since $u'$ and $v'$ must already be in different level $i$ modules of $H$).
If $U$ and $v$ are in different level $i$ modules of $G$,
let the level $i$ module of $G$ containing $u$ be
$M_u$ and the one containing $v$ be $M_v$.
 Let $M_{u'}$ be the level $i$ module containing $u'$
and $M_{v'}$ be the one containing~$v'$.
We make a case distinction.
\begin{enumerate}
\item Assume $\text{twins}(G / M_i)$ is not the trivial partition
of $G / M_i$ into singletons.
Then $M_u$ and $M_v$ are twins in $G/M_i$,
which means $u$ and $v$ agree on all neighbors in $G$ outside of $M_u \cup M_v$.
Note that $M_{u'}$ and $M_{v'}$ are not twins in the corresponding quotient graph,
and so $u'$ and $v'$ disagree on a neighbor outside of $M_{u'} \cup M_{v'}$.
Let $w \in V(H)$ be that vertex.
By induction, if $M_u$ is not mapped to $M_{u'}$
or $M_v$ is not mapped to $M_{v'}$,
then Spoiler has a winning strategy by pebbling a vertex that gets mapped outside
of its corresponding module.
If $M_u$ is mapped to $M_{u'}$
and $M_v$ is mapped to $M_{v'}$ Spoiler can now win in the position $uv,u'v'$
by pebbling the preimage of $w$ because $u$ and $v$
agree on all their neighbors outside their modules, and $u'$ and $v'$
disagree on $w$. See \Cref{fig:modulescase1}.

\begin{figure}
			\centering
			\begin{tikzpicture}[minimum size = 1em,
    				baseline=(current bounding box.north),
    				every path/.style={thick}]
    			\node (G) at (-2,3) {$G:$};
   			\node (Mu) at (0,3) {$M_u$};
   			\node (Mv) at (0,-1) {$M_v$};
			\begin{scope}[every node/.style={circle,draw}]
   			\node (v) at (0,0) {$v$};
   			\node (u) at (0,2) {$u$};
   			\node (w) at (2,1) {};
   			\node[ellipse,draw=black,inner sep=0,
    				fit = (Mu) (u)] (Mubub) {};
   			\node[ellipse,draw=black,inner sep=0,
    				fit = (Mv) (v)] (Mvbub) {};
			\end{scope}
    			\node (H) at (5,3) {$H:$};
   			\node (Mup) at (7,3) {$M_{u'}$};
   			\node (Mvp) at (7,-1) {$M_{v'}$};
			\begin{scope}[every node/.style={circle,draw}]
   			\node (up) at (7,2) {$u'$};
   			\node (vp) at (7,0) {$v'$};
   			\node (wp) at (5,1) {$w$};
   			\node[ellipse,draw=black,inner sep=0,
    				fit = (Mup) (up)] (Mupbub) {};
   			\node[ellipse,draw=black,inner sep=0,
    				fit = (Mvp) (vp)] (Mvpbub) {};
			\end{scope}
			\draw[dashed] (wp) -- (up);
			\draw[dashed] (wp) -- (vp);
			\draw (w) -- (u);
			\draw (w) -- (v);
			\path[->, draw=red] (w) edge (wp);
			\path[->, draw=blue] (Mubub) edge (Mupbub);
			\path[->, draw=blue] (Mvbub) edge (Mvpbub);
		\end{tikzpicture}
	    	\caption{Illustration of Case 1 of the proof for
	    	\Cref{2wl-modules}.
    		The bijection chosen by Duplicator is indicated by the arrows.}
    		\label{fig:modulescase1}
    	\end{figure}
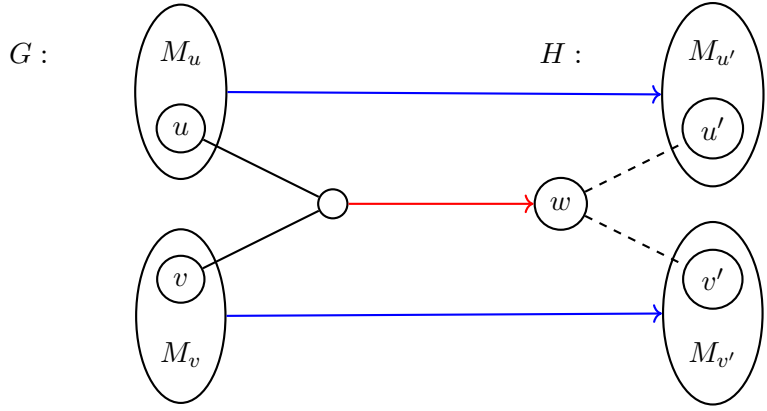

\item Otherwise, $\text{twins}(G / M_i)$ is the trivial partition
of $G / M_i$ into singletons.
Let $M^H_{i}$ be the partition of $H$
into modules level $i$.
If $\text{twins}(H / M^H_i)$ is not the trivial partition
of $H / M^H_i$ into singletons, then return to the argument in Case 1
with $G$ and $H$ exchanged.
So assume it is.
Then,~$M_u$ and $M_v$ are part of the same connected and coconnected
component of $G / M_i$.
Also, $M_{u'}$ and $M_{v'}$ are not part of the same
connected and coconnected component
of $H / M^H_i$.
Because the sets of $M_i$ are modules, paths between modules in $G/M_i$ and $H/M^H_i$
yield paths between their vertices in $G$ and $H$.
Therefore, either there is a path from $u$ to $v$
in~$G$ but not from $u'$ to $v'$ in $H$,
or there is a path from $u$ to $v$
in $\overline{G}$ but not from $u'$ to $v'$ in $\overline{H}$.
In both cases, Spoiler can win by using a third pebble to trace the path
in $G$ or $\overline{G}$ and win. See \Cref{fig:modulescase2}. \qedhere

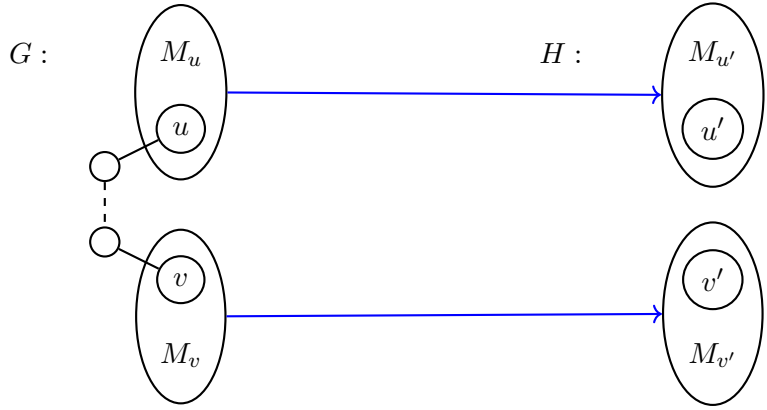
\begin{figure}
			\centering
			\begin{tikzpicture}[minimum size = 1em,
    				baseline=(current bounding box.north),
    				every path/.style={thick}]
    			\node (G) at (-2,3) {$G:$};
   			\node (Mu) at (0,3) {$M_u$};
   			\node (Mv) at (0,-1) {$M_v$};
			\begin{scope}[every node/.style={circle,draw}]
   			\node (v) at (0,0) {$v$};
   			\node (u) at (0,2) {$u$};
   			\node (p1) at (-1,1.5) {};
   			\node (p2) at (-1,0.5) {};
   			\node[ellipse,draw=black,inner sep=0,
    				fit = (Mu) (u)] (Mubub) {};
   			\node[ellipse,draw=black,inner sep=0,
    				fit = (Mv) (v)] (Mvbub) {};
			\end{scope}
    			\node (H) at (5,3) {$H:$};
   			\node (Mup) at (7,3) {$M_{u'}$};
   			\node (Mvp) at (7,-1) {$M_{v'}$};
			\begin{scope}[every node/.style={circle,draw}]
   			\node (up) at (7,2) {$u'$};
   			\node (vp) at (7,0) {$v'$};
   			\node[ellipse,draw=black,inner sep=0,
    				fit = (Mup) (up)] (Mupbub) {};
   			\node[ellipse,draw=black,inner sep=0,
    				fit = (Mvp) (vp)] (Mvpbub) {};
			\end{scope}
			\draw (u) -- (p1);
			\draw (v) -- (p2);
			\draw[dashed] (p1) -- (p2);
			\path[->, draw=blue] (Mubub) edge (Mupbub);
			\path[->, draw=blue] (Mvbub) edge (Mvpbub);
		\end{tikzpicture}
	    	\caption{Illustration of Case 2 of the proof for
	    	\Cref{2wl-modules}.
    		The bijection chosen by Duplicator is indicated by the arrows.}
    		\label{fig:modulescase2}
    	\end{figure}

\end{enumerate}
\end{itemize}
\end{proof}

This Lemma now allows us to establish
the key result of this section,
which concerns colored quotient prime graphs
where non-isomorphic modules get a different color.

\begin{definition}
Let $<^*$ be some arbitrary but fixed order
on the isomorphism types of colored graphs.
Given a colored graph $G_c = (G,c)$,
let $I \coloneqq \{\itp(G_c[M]) \mid M \in M_G\}$
be the set of isomorphism types of the colored maximal modules in $G$.
Let $(i_1,\dots,i_k)$ be an enumeration of $I$
such that $i_1 <^* \dots <^* i_k$.
We define the graph $G_c^* \coloneqq (G/M_G,c^*)$
where $c^*(M) = k$ such that $\itp(G_c[M]) = i_k$.
\end{definition}

Since the order of the colors does not matter
when describing a strategy for Spoiler,
we can now show this general statement.

\begin{lemma}\label{2wl-modules2}
Let $G_c = (G,c)$ be a colored graph such that
for its uncolored version $G$,
the quotient graph $G/M_G$ is prime, and let $k \coloneqq |M_G|$
and $M_G = \{M_1,\ldots,M_k\}$.
Then
\[\dimwl(G_c) \leq \max \, \{ \dimwl(G_c^*), \dimwl(G_c[M_1]),\ldots,\dimwl(G_c[M_k]), 2 \}.\]
\end{lemma}

\begin{proof}
Let $G_c,H_c$ be non-isomorphic colored graphs
and $G,H$ be their underlying uncolored graphs.
Given a vertex $v$ in $G$ or $H$, we denote by $M_v$ the maximal module it belongs to
in the corresponding graph.
Now, play $\bpk(G_c,H_c)$.
\begin{itemize}
	\item If at any point in the game Duplicator chooses a bijection $b$
	such that for some vertex $v$, it holds that $G_c[M_v] \not \cong H_c[M_{b(v)}]$,
	respond by placing a pebble on $v$.
	For the rest of the game, as long as some vertex in $M_v$ is pebbled,
	Duplicator must choose bijections $b'$ such that
	$b'(M_v) = M_{b(v)}$,
	according to \Cref{2wl-modules}.
	So since $G_c[M_v] \not \cong H_c[M_{b(v)}]$
	and Duplicator must always choose bijections $b'$ such that
	$b'(M_v) = b'(M_{b(v)})$,
	Spoiler can now move to the strategy on that module and win.
	\item
	If Duplicator chooses a bijection $b$
	such that for for all vertices~$v$,
	it holds that $G_c[M_v] \cong H_c[M_{b(v)}]$,
	note that $G_c^* \not \cong H_c^*$ because $G_c \not \cong H_c$.
	Therefore, Spoiler can play a winning game on $G_c^*$ and $H_c^*$.
	We can translate this into a winning strategy for Spoiler on $G_c$ and $H_c$
	using the same number of pebbles.
	Pebbles in a module $M$ of $G_c$ or $H_c$
	correspond to a pebble placed on $M$ in $G_c^*$ or $H_c^*$.
	If Duplicator chooses a bijection $b$ on $G_c$ and $H_c$
	we can translate it into a bijection $b^*$ on $G_c^*$ and $H_c^*$
	as follows: Let
	\[R \coloneqq \{(M_v,M_w) \mid v \in V(G_c), w \in V(H_c),
	b(v) = w \}.\]
	Then we can pick a bijection $b^*$ such that if $b^*(M) = M'$,
	then ${(M,M') \in R}$.
	Picking such a bijection is possible because
	each vertex in a maximal module~$M$ is mapped to a vertex
	in a maximal module $M'$ of equal size.
	Spoiler will have a response to this bijection
	by placing pebbles on modules in $G_c^*$ and $H_c^*$,
	which correspond to pebbles placed anywhere inside the modules in $G_c$ and $H_c$.
	Playing the game on $G_c^*$ and $H_c^*$ in this way,
	$b^*$ respects the existing pebbles and colors,
	and if Spoiler eventually reaches a winning position on $G_c^*$ and $H_c^*$,
	it is also a winning position on $G_c,H_c$,
	because two modules are adjacent in $G_c^*$ or $H_c^*$
	if and only if all of the vertices in the modules are adjacent in $G_c$ or $H_c$. \qedhere
\end{itemize}
\end{proof}

We are now ready to show the central result of this section.

\begin{lemma}
3{\nobreakdash-}WL identifies
colored graphs of twin{\nobreakdash-}width~1.
\end{lemma}
\begin{proof}
Let $G$ be a colored graph of twin-width 1.
We show by induction over $\ModTree_G = (\mathcal{M},E,\ell)$
that 3-WL identifies $G$.
Since 3-WL identifies 1-vertex graphs it identifies all leaf modules labeled $\texttt{single}$.
Now let $M \in \mathcal{M}$ be a module such that 3-WL identifies $G[M']$ for all children $M'$ of $M$
in $\ModTree_G$.
If $M$ is labeled $\texttt{parallel}$, then $G[M]$ is the disjoint union of its children 
 and, hence, can be identified by 3-WL.
If $M$ is labeled $\texttt{series}$, then $G[M]$ can be identified by 3-WL since
 3-WL identifies a graph precisely if it identifies its complement.
If $M$ is labeled $\texttt{prime}$, then
$G[M]^*$ is a prime colored graph of twin-width 1,
and thus identified by 3-WL by \Cref{3wl-coloredprimetww1}.
3-WL also identifies all induced subgraphs of modules in $G[M]/M_{G[M]}$.
Therefore, by \Cref{2wl-modules2},
3-WL also identifies $G[M]$. 
By induction, 3-WL identifies $G[M]$ for all $M \in \mathcal{M}$.
Since the root module is~$V(G)$,
3-WL identifies $G$.
\end{proof}

It is not clear if the fourth pebble in the proof for \Cref{3wl-coloredprimetww1} is actually necessary.
Therefore:

\begin{theorem}
The class of colored graphs of twin-width 1 has WL-dimension 2 or 3.
\end{theorem}

Finally, note that \cite{rank-width} also shows that
for all $k$, every polynomial-time decidable property
of the class of all
graphs of rank width at most $k$
is expressible in FP+C,
because it admits FP+C-definable
canonization.
The class of graphs of twin-width 1 is such a class,
since it has rank-width bounded by 2.
However, the construction in \cite{rank-width}
uses $2^{2k}$-ary relation symbols
where $k$ is the rank-width of the graph.
There is a method to show definable canonization
for graphs of twin-width 1 which uses only
at most 5-ary relation symbols.
The runtime is $\mathcal{O}(n^5)$
and it works similarly to the 
canonization algorithms described in this section.
A description can be found in Appendix \ref{tww1:definablecanonization}.

\section{Stable classes of twin-width 2}
As exploited in the previous section,
graphs of twin-width \(1\) always admit a contraction sequence which only ever contains a single red edge.
In particular, they have bounded component twin-width and thus also bounded rank- and clique-width.
This is not true for graphs of twin-width \(2\), since unit interval graphs have twin-width at most \(2\) \cite{tww1}
but unbounded rank-width \cite{cw_grids}. Still, it was shown in \cite{sparse_tww2_bounded_tw} that sparse graphs of
twin-width \(2\) have bounded tree-width, with a polynomial bound on the tree-width in terms of the size of a forbidden
semi-induced biclique.
We strengthen and extend this result by showing that graph classes of twin-width \(2\) which are \emph{monadically stable},
i.e., forbid some half-graph \(H_t\) as a semi-induced subgraph, have their rank-width bounded by a linear function in the size of the forbidden semi-induced half-graph. This in turn also gives a linear upper bound on their Weisfeiler-Leman dimension.

We strengthen and extend this result by showing that the rank-width of graphs of twin-width \(2\) is linearly bounded in the size of their largest semi-induced half-graph.
In particular, \emph{monadically stable} graph classes of twin-width \(2\), that is,
classes of twin-width \(2\) which forbid some semi-induced half-graph, have bounded rank-width.

\subsection{Cuts induced by red edges in graphs of twin-width 1}
In contraction sequences of twin-width \(2\), all red components are either red paths or cycles. When performing contractions involving inner vertices of such a red path or cycle, these contractions essentially behave like width-1 contractions, in the sense
that they do not create additional red edges. This motivates our approach:
We first analyze the structure that emerges from \(1\)-contraction sequences, and then apply this understanding to the locally width-1 behaviour within general \(2\)-contraction sequences. During this, it will often be beneficial to talk about
partition sequences instead of contraction sequences.

We thus start our investigation by returning to graphs of twin-width \(1\) and consider the following question:
Given a graph \(G\) with a \(1\)-partition sequence
\((\mathcal{P}_i)_{i\in[n]}\) and a red edge \(PQ\in R(G/\mathcal{P}_i)\),
what can we say about the graph \(G[P,Q]\)? Equivalently, which bipartite graphs
\(G[A,B]\) admit a \(1\)-partition sequence whose final two parts are the two sides \(A\) and~\(B\)?

One example of such bipartite graphs are half-graphs:
\begin{lemma}\label{lem:half-graph:tww1}
For every \(t\in\N\), the half-graph \(H_t\) admits a \(1\)-partition sequence whose final
two parts are the two sides of the half-graph.
\begin{proof}
Let \(H_t'\) be the trigraph obtained from \(H_t\) by making the edge \(v_tw_t\) red.
Now, observe that first contracting the pair \(w_tw_{t-1}\) and then contracting the pair
\(v_tv_{t-1}\) does not create vertices of red degree \(2\), and the resulting graph is isomorphic
to \(H_{t-1}'\). Thus, we can inductively find a \(1\)-partition sequence contracting \(H_t'\)
to \(H_1'\), which is a single red edge.
\end{proof}
\end{lemma}

Because twin-width is monotone with respect to induced subgraphs, it follows that also every
induced subgraph of a half-graph admits such a \(1\)-partition sequence. We call such graphs
\emph{partial half-graphs}.
We show that this is already the complete picture, that is, every bipartite graph induced by a red edge
in a \(1\)-partition sequence is a partial half-graph.

\begin{lemma}\label{lem:tww1:red_cuts}
For every graph \(G\), every \(1\)-partition sequence \((\mathcal{P}_i)_{i\in[n]}\) of \(G\),
and every two distinct parts \(P,Q\in\mathcal{P}_i\), the graph \(G[P,Q]\) is a partial half-graph.
\begin{proof}
We prove the claim by backwards induction on \(i\). If \(i=n\), all parts are singletons and the claim is clear,
because both edges and non-edges are partial half-graphs.
Thus, assume that the claim is true for all \(j>i\) and pick two distinct parts \(P,Q\in\mathcal{P}_i\).
If \(P,Q\in\mathcal{P}_{i+1}\), the claim is immediate. Thus, assume w.l.o.g. that \(Q=Q_0\dotcup Q_1\)
with \(Q_0,Q_1\in\mathcal{P}_{i+1}\). Because \(P\) has at most one non-homogeneous neighbor in \(\mathcal{P}_{i+1}\),
either \(Q_0\) or \(Q_1\), say \(Q_1\) is homogeneously connected to \(P\), while \(G[P,Q_0]\) is a partial half-graph
by induction, say \(G[P,Q_0]\subseteq H_k\). Now, if \(P\) is fully connected to \(Q_1\), we claim that \(G[P,Q_0\dotcup Q_1]\subseteq H_{k+|Q_1|}\) by embedding \(H_k\) into the bottom part of this larger half-graph, and embedding \(Q_1\)
arbitrarily in the top \(|Q_1|\)-many elements of the respective part in \(H_{k+|Q_1|}\).
If \(P\) is disconnected to \(Q_1\), we instead embed \(H_k\) into the top part of the half-graph \(H_{k,k+|Q_1|}\)
and embed \(Q_1\) into the smallest \(|Q_1|\)-many elements of the respective part in \(H_{k+|Q_1|}\).
\end{proof}
\end{lemma}

Equipped with this characterization of bipartite graphs induced by red edges in graphs of twin-width \(1\),
we next look at what kinds of subgraphs are forced by high connectivity between two parts.

\begin{lemma}\label{lem:partial_half-graph_has_rank=ladder-index}
For every \(k\in\N\), a partial half-graph \(H[L,R]\) contains an induced half-graph \(H_t\) if and only if \(\rk_H(L,R)\geq t\).
\begin{proof}
Because deleting twins and isolated vertices neither changes the rank nor affects the existence
of induced half-graphs, we may assume that \(H\) contains neither twins nor isolated vertices.
Since the vertices of either side of a half-graph are linearly ordered by inclusion of their neighborhoods,
the vertices of either side of a partial half-graph are linearly pre-ordered by inclusion of their neighborhoods.
But since any non-trivial equivalence class in this ordering would be a twin-class, the vertices of either side
of \(H\) are also linearly ordered by inclusion of their neighborhoods, and we can write
\(L=\{\ell_1,\dots,\ell_n\}\) and \(R=\{r_1,\dots,r_m\}\) for some \(n,m\in\N\),
such that \(N_H(\ell_i)\subseteq N_H(\ell_j)\) if and only if \(i\geq j\) and \(N_H(r_i)\subseteq N_H(r_j)\)
if and only if \(i\leq j\). Since neither side contains an isolated vertex, the vertices \(\ell_1\) must
be fully connected to \(R\), while \(\ell_n\) still has at least one neighbor. Since the neighborhoods
of \(\ell_i\) and \(\ell_j\) must differ for \(i\neq j\), this implies \(n\leq m\) and by the same argument
with \(L\) and \(R\) exchanged that \(n=m\), and \(N_H(\ell_i)=\{r_i,\dots,r_n\}\). Thus, \(H\) is
isomorphic to a standard half-graph \(H_n\). But this graph has rank exactly \(n\).
\end{proof}
\end{lemma}

Similarly, we can also relate the size of a maximal matching to the size of a maximal biclique.
\begin{lemma}\label{lem:partial_half-graph_has_matchings<=2biclique}
For every \(k\in\N\), if a partial half-graph \(H[L,R]\) contains a matching of size \(2t-1\), then \(H\) contains a biclique \(K_{t,t}\).
\begin{proof}
Let \(\iota\colon H\to H_n\) be an embedding, and label the vertices of \(L\) as
\(\ell_1,\dots,\ell_{|L|}\) according to their order in the embedding, and the vertices of \(R\) as
\(r_1,\dots,r_{|R|}\) in the reverse order of their embedding. This guarantees that the two enumerations
of \(L\) and \(R\) are both ordered according to inclusion of their neighborhoods, with \(\ell_1\) and \(r_1\)
having maximal neighborhoods.

If the graph \(H\) is \(K_{k,k}\)-free, then in particular, the elements \(\ell_1,\dots,\ell_k\) are not fully
connected to the elements \(r_1,\dots,r_k\), which implies \(\iota(\ell_k)>\iota(r_k)\).
But this in particular implies that every edge of \(H\) is incident to the vertex set \(\{\ell_1,\dots,\ell_{k-1},r_1,\dots,r_{k-1}\) of size \(2k-2\), which forbids a matching of size \(2k-1\).
\end{proof}
\end{lemma}

\subsection{Stable graphs of twin-width 2 have bounded rank-width}
With these preparations, we are ready to proceed to show that \(H_k\)-semi-free graphs of twin-width~\(2\) have rank-width \(O(k)\). Our general strategy is similar to that of \cite{sparse_tww2_bounded_tw} and is as follows:
Assuming \(G\) has large rank-width, \Cref{thm:prelims:rw_well-linked_sets} guarantees the existence of a large
\(\rk\)-well linked set. Investigating the \(2\)-partition sequence on this well-linked set, we can extract a single red edge
inside a longer red path which has good connectivity properties. Now, going back through the partition sequence, we note that as long as the red path
persists, the (un-)contraction sequence applied to the red edge of this path is actually a \(1\)-partition sequence,
since the extra red edges of the path forbid additional red edges between the two parts in question. Thus,
we find a large bipartite graph of twin-width \(1\) with good connectivity properties, which by the previous section
implies the existence of a large half-graph.

We start by showing that graphs of twin-width \(2\) which contain a large well-linked set must also contain a highly connected red path at some point along the partition sequence.
\begin{lemma}\label{lem:stable:existence_of_highly-connected_path}
Let \(t\geq 0\) be an integer and \(k\coloneqq 2t+9\).
If \(G\) is a graph which contains a \(\rk\)-well-linked set \(W\) of size \(11k-5\) and
\((\mathcal{P}_i)_{i\in[n]}\) is a \(2\)-partition sequence of \(G\), then
there exists some \(i\in[n]\) and four parts \(X_1,X_2,X_3,X_4\in\mathcal{P}_i\) such that
\begin{itemize}
\item \(X_1X_2X_3X_4\) forms a red \(4\)-path in this order in \(G/\mathcal{P}_i\),
\item there is no red edge \(X_1X_4\),
\item \(\kappa^{\rk}_{G[X_1\cup X_2\cup X_3\cup X_4]}(X_1,X_4)\geq t\).
\end{itemize}
\begin{proof}
Let \(i\in[n]\) be minimal such that every part of \(\mathcal{P}_i\) contains at most \(2k-1\) vertices of \(W\).
By minimality, \(\mathcal{P}_i\) contains a part \(Z\) which contains at least \(k\) vertices of \(W\).

For some \(d\in\N\) and \(R\in\{=,<,>,\leq,\geq\}\), let \(N^{R d}_{\red}(Z)\) be the set of parts in \(\mathcal{P}_i\) with the specified
distance from \(Z\) in the red graph of \(G/\mathcal{P}_i\).
By abuse of notation, we also denote the set of vertices of \(G\) contained in those parts by \(N^{R d}_\red(Z)\).

Because the red graph has maximum degree at most \(2\), \(N^{\leq d}_{\red}(Z)\) contains at most \(2d+1\) parts
and \(N^{=d}_{\red}(Z)\) contains at most \(2\) parts.
In particular, \(N^{\leq 2}_{\red}\) consists of at most five parts, which together contain at most \(10k-5\) vertices from \(W\).
In particular, the set \(W_3\coloneqq W\cap\bigcup N^{\geq 3}(Z)\) contains at least \(k\) vertices.

Thus, because \(W\) is well-linked, we find
\[\kappa^{\rk}_G(Z,W_3)\geq\min\{|Z|,|W_3|\}\geq k.\]

Now, every part of \(\mathcal{P}_i\) which intersects \(W_3\) is either one of the at most two parts in \(N^{=3}_{\red}(Z)\),
or is homogeneously connected to all parts in \(N^{\leq 2}_{\red}(Z)\). But the latter parts together have rank at most \(5\) to \(N^{\leq2}_{\red}(Z)\),
which, using \Cref{lem:properties_of_rk*}(b) implies that
\[\kappa^{\rk}_{G\left[ N^{\leq 3}_{\red}(Z)\right]}\left(Z, N^{=3}_{\red}(Z)\right)\geq k-5=2t+4.\]
In particular, \(N^{=3}_{\red}(Z)\) is non-empty.

First assume that it contains two parts,
which means that there is a red \(7\)-path \(L_3\)-\(L_2\)-\(L_1\)-\(Z\)-\(R_1\)-\(R_2\)-\(R_3\).
We claim that at least one of the two \(4\)-paths \(L_3\)-\(L_2\)-\(L_1\)-\(Z\) or
\(Z\)-\(R_1\)-\(R_2\)-\(R_3\) fulfills the requirements of the claim.
Because the first two properties are clear, we may assume for the sake of contradiction
that both \(\kappa^{\rk}_{G[Z\cup R_1\cup R_2\cup R_3]}(Z,R_3)\leq t-1\)
and \(\kappa^{\rk}_{G[Z\cup L_1\cup L_2\cup L_3]}(Z,L_3)\leq t-1\).
But this means that there exists two cuts \(L_1\cup L_2=L_Z\dotcup L_P\)
and \(R_1\cup R_2=R_Z\dotcup R_P\) such that
\(\rk_G(Z\cup L_Z,L_3\cup L_P)\leq t-1\) and
\(\rk_G(Z\cup R_Z,R_3\cup R_P)\leq t-1\).
But then,
\begin{align*}
2t+4
 &\leq \kappa^{\rk}_{G[ N^{\leq 3}_{\red}(Z)]}(Z,L_3\cup R_3)\\
 &\leq \rk_G(Z\cup L_Z\cup R_Z,L_P\cup R_P\cup L_3\cup R_3)\\
 &\leq \rk_G(Z\cup L_Z,L_P\cup L_3)+\rk_G(Z\cup R_Z,R_P\cup R_3)\\
 &\quad	+\rk_G(L_Z,R_3\cup R_P)+\rk_G(R_Z,L_3\cup L_P)\\
 &\leq 4+2(t-1)=2t+2,
\end{align*}
which is a contradiction.

The case that instead of a \(7\)-path we get a red \(6\)-cycle \(P\)-\(L_2\)-\(L_1\)-\(Z\)-\(R_1\)-\(R_2\)-\(P\)
is analogous to the first case with the only difference being that we included \(\rk_G(Z,P)\) twice in the above calculation.
Adding it once more in order to cancel it out, we find that
\[2t+4\leq \kappa^{\rk}_{G[ N^{\leq 3}_{\red}(Z)]}(Z,P)\leq 2t+3,\]
which is again a contradiction.

Thus, we are left with the case that there is only one red \(4\)-path starting at \(Z\), say it is \(Z\)-\(R_1\)-\(R_2\)-\(P\),
and let \(L\) be the (possibly empty) union of parts which form the second red path attached to \(Z\). Because \(L\) consists of
at most two parts, and the only outgoing red edge connects it to \(Z\), we get \(\rk(L,R_1\cup R_2\cup P)\leq 2\).
Thus, we find 
\begin{align*}
\kappa^{\rk}_{G[Z\cup R_1\cup R_2\cup P]}(Z,P)
&\geq\kappa^{\rk}_{G[Z\cup R_1\cup R_2\cup P\cup L]}(Z\cup L,P)-\kappa^{\rk}_{G[R_1\cup R_3\cup P\cup L]}(L,R_1\cup R_2\cup P)\\
&\geq\kappa^{\rk}_{G[N^{\leq 3}_{\red}(Z)]}(Z,P)-2\\
&\geq 2t+2>t.\qedhere
\end{align*}
\end{proof}
\end{lemma}

Next, we want to show that the red \(4\)-path whose existence we proved in the previous lemma
forces the existence of a large semi-induced half-graph.
\begin{lemma}\label{lem:stable:highly-connected_path_implies_half-graph}
Let \(G\) be a graph with a \(2\)-partition sequence \((\mathcal{P}_i)_{i\leq n}\) of \(G\),
and assume for some \(m\in[n]\), there exists four parts \(X_1,X_2,X_3,X_4\in\mathcal{P}_m\) such that
\begin{itemize}
\item \(X_1\), \(X_2\), \(X_3\), \(X_4\) form a red path in this order in \(G/\mathcal{P}_m\),
	and there is no red edge \(X_1X_4\),
\item \(\kappa^{\rk}_{G[X]}(X_1,X_4)\geq t+6\),
\end{itemize}
where \(X\coloneqq X_1\cup X_2\cup X_3\cup X_4\).
Then \(G\) contains a semi-induced half-graph \(H_t\).
\begin{proof}
Let \(i>m\) be minimal such that there do not exist four parts \(Y_1,Y_2,Y_3,Y_4\in\mathcal{P}_i\)
with the following properties:
\begin{itemize}
\item the four parts \(Y_1\), \(Y_2\), \(Y_3\), and \(Y_4\) form a red path in this order in the trigraph \(G/\mathcal{P}_i\), and there is no red edge \(Y_1Y_4\),
\item \(Y_2\subseteq X_2\) and \(Y_3\subseteq X_3\).
\end{itemize}
Note that such an \(i\) exists, because the graph \(G=G/\mathcal{P}_{|G|}\) contains no red edges
and thus surely does not contain a red \(4\)-path.

\begin{claim}
For every \(i>j\geq m\), there is a single red edge crossing the cut \((X_2,X_3)\) in \(G/\mathcal{P}_j\),
which is the central edge in a red \(4\)-path.
\begin{claimproof}
The claim is true for \(j=m\), where \(X_2X_3\) itself is this singe red edge, which is contained in the red path \(X_1X_2X_3X_4\).
Now, assume the claim is true for some \(j\) with \(i-1>j\geq m\), with red \(4\)-path \(Y_1Y_2Y_3Y_4\).
If \(Y_2\) and \(Y_3\) are still parts of \(\mathcal{P}_{j+1}\), then \(Y_2Y_3\) is still the unique
red edge crossing the cut \((X_2,X_3\) in \(G/\mathcal{P}_{j+1}\),
since uncontracting any part which is not incident to a red edge crossing the cut \((X_2,X_3)\) can never
create or remove a new red edge crossing this cut. Moreover, by minimality of \(i\), there must still exist a red
\(4\)-path \(Y_1'Y_2'Y_3'Y_4'\) whose central edge is crosses the cut \((X_2,X_3)\), which proves the claim for \(j+1\).
Thus assume w.l.o.g. that \(Y_2\notin\mathcal{P}_{j+1}\), meaning that \(Y_2\) is obtained by contracting two
parts \(Y_2^{(1)}\) and \(Y_2^{(2)}\) in \(\mathcal{P}_{j+1}\). But since \(Y_3\) still has the red neighbor \(Y_4\),
it cannot share a red edge with both \(Y_2^{(1)}\) and \(Y_2^{(2)}\). Thus, there is still at most one red edge
crossing the cut \((X_2,X_3)\). But since minimality of \(i\) guarantees the existence of a red \(4\)-path
whose central edge crosses the cut \((X_2,X_3)\), the claim follows.
\end{claimproof}
\end{claim}

Because \(i\) is minimal, we know that \(G/\mathcal{P}_{i-1}\) does contain such a red \(4\)-path
\(Y_1Y_2Y_3Y_4\). In \(G/\mathcal{P}_i\), one of these four parts, say \(Y_i\)
splits into two parts \(Y_i^1\) and \(Y_i^2\) such that either \(i\in\{2,3\}\) and \(Y_i^1\) and \(Y_i^2\)
share no red edge, or one of the red neighbors \(Y_j\) of \(Y_i\) in \(G/\mathcal{P}_{i-1}\) is not incident
via a red edge to either of the parts \(Y_i^1\) and \(Y_i^2\).
In the former case, we say that the part \(Y_i\) is broken, while in the latter case, we say that
the edge \(Y_iY_j\) is broken.
In this case, we say that the red edge \(Y_iY_j\) is broken. Note that the edge \(Y_iY_j\) being broken
implies \(\rk_G(Y_i,Y_j)=1\).

By symmetry, we may assume that either \(Y_3\) is broken, or one of the edges \(Y_2Y_3\) or \(Y_3Y_4\) is broken.

If the red edge \(Y_2Y_3\) is broken, then consider the bipartite graph \(G[X_2,X_3]\).
\begin{claim}
The bipartite graph \(G[X_2,X_3]\) admits a \(1\)-partition sequence whose final two parts are \(X_2\) and \(X_3\).
\begin{claimproof}
Consider the partition sequence induced by the partition sequence \((\mathcal{P}_j)_{m\leq j\leq n}\) on \(X_2\cup X_3\).
By the previous claim, there exists exactly one red edge crossing the cut \((X_2,X_3)\) in \(G/\mathcal{P}_j\) for all
\(m\leq j<i\). Since the edge \(Y_2Y_3\) is broken, there is no red edge crossing \((X_2,X_3)\) in \(G/\mathcal{P}_i\),
and thus there is also no red edge crossing \((X_2,X_3)\) in \(G/\mathcal{P}_j\) for any \(j\geq i\).
Thus, the partition sequence induced on \(G[X_2,X_3]\) contains at most a single red edge in every step.
\end{claimproof}
\end{claim}
Thus, \(G[X_2,X_3]\) is a partial half-graph by \Cref{lem:tww1:red_cuts}.
Because furthermore
\[\rk_G(X_2,X_3)\geq\rk_G(X_1\cup X_2,X_3\cup X_4)-2\geq\kappa^{\rk}_{G[X]}(X_1,X_4)>t,\]
\Cref{lem:partial_half-graph_has_rank=ladder-index} implies that \(G[X_2,X_3]\) contains
a semi-induced half-graph \(H_t\).

If instead the edge \(Y_3Y_4\) is broken, consider the graph \(G[X_2,X_3\setminus Y_3]\). Because there is no longer
a red edge crossing the cut \((X_2,X_3\setminus Y_3)\) in \(G/\mathcal{P}_i\),
we can argue just as in the previous claim that there exists a \(1\)-partition sequence
of \(G[X_2,X_3\setminus Y_3]\) whose final two parts are \(X_2\) and \(X_3\setminus Y_3\).

Moreover, because \(Y_3\) splits into at most two parts in \(\mathcal{P}_i\),
both of which are homogeneously connected to all other parts of \(X_3\),
we have \(\rk_G(Y_3,X_3\setminus Y_3)\leq 2\).
But then,
\begin{align*}
\rk_G(X_2,X_3\setminus Y_3)
&\geq \rk_G(X_2\cup Y_3,X_3\setminus Y_3)-2\\
&\geq \rk_G(X_1\cup X_2\cup Y_3,(X_3\setminus Y_3)\cup X_4)-2\\
	&\quad-\rk_G(X_1,X_4\cup X_3\setminus Y_3)-\rk_G(X_4,X_1\cup X_2)-\rk_G(X_4,Y_3)\\
&\geq \kappa^{\rk}_{G[X]}(X_1,X_4)-6\geq t.
\end{align*}
Thus, \Cref{lem:partial_half-graph_has_rank=ladder-index} again implies
that \(G[X_2,X_3\setminus Y_3]\) contains a semi-induced half-graph \(H_t\).

Finally, assume that the part \(Y_3\) is broken, and that \(Y_2\) is not incident via a red
edge to the part \(Y_3^2\). Then, we consider the bipartite graph \(G[X_2,X_3\setminus Y_3^1]\),
which again no longer contains a red edge with respect to the partition \(\mathcal{P}_i\).
Analogous to before, we obtain a \(1\)-partition sequence whose final two parts are \(X_2\) and \(X_3\setminus Y_3^1\).
Further, because \(Y_1Y_2Y_3^1Y_4\) cannot be a red \(4\)-path by assumption, either the edge \(Y_2Y_3\) is broken,
in which case we are done, or \(Y_3^1Y_4\) is not a red edge. But then, \(Y_3^1\)
shares no red edge with any part in \(X_3\) besides itself, and thus has rank at most \(1\) to \(X_3\setminus Y_3^1\).
Thus,
\begin{align*}
\rk_G(X_2,X_3\setminus Y_3^1)
&\geq \rk_G(X_2\cup Y_3^1,X_3\setminus Y_3^1)-1\\
&\geq \rk_G(X_1\cup X_2\cup Y_3^1,(X_3\setminus Y_3^1)\cup X_4)-1\\
	&\quad-\rk_G(X_1,X_4\cup X_3\setminus Y_3^1)-\rk_G(X_4,X_1\cup X_2)-\rk_G(X_4,Y_3^1)\\
&\geq \kappa^{\rk}_{G[X]}(X_1,X_4)-4>t,
\end{align*}
which again implies using \Cref{lem:partial_half-graph_has_rank=ladder-index} that
\(G[X_2,X_3\setminus Y_3^1]\) contains a semi-induced half-graph~\(H_t\).
\end{proof}
\end{lemma}

Combining the previous two lemmas yields our main theorem:
\begin{theorem}\label{thm:stable:tww2->bounded_rw}
Every \(H_t\)-semi-free graph \(G\) with twin-width at most \(2\) has rank-width at most \(22t+170\).
In particular, every monadically stable graph class of twin-width at most \(2\) has bounded rank-width.
\begin{proof}
Assume the claim were false, that is, \(\rw(G)>22t+170\).
By \Cref{thm:prelims:rw_well-linked_sets}, the graph \(G\) then contains a \(\rk\)-well-linked set of size
\(22t+170=11k-5\) with \(k=2(t+6)-9\). Thus, by \Cref{lem:stable:existence_of_highly-connected_path},
we find a highly connected red path, which by \Cref{lem:stable:highly-connected_path_implies_half-graph}
implies the existence of an \(H_t\). But this contradicts our assumption.

The second claim follows from the observation that every monadically stable class is \(H_t\)-semi-free for some \(t\).
\end{proof}
\end{theorem}

This theorem generalizes the related result of \cite{sparse_tww2_bounded_tw} that every
sparse graph class of twin-width at most \(2\) has bounded tree-width,
because every sparse class of bounded rank-width also has bounded tree-width.
But while this approach gives worse bounds on the tree-width than given in \cite{sparse_tww2_bounded_tw},
we can actually improve their bounds by replacing their use of the excluded grid theorem for tree-width
by the use of \(\mm\)-well-linked sets as in our proof and using \Cref{lem:partial_half-graph_has_matchings<=2biclique}
instead of \Cref{lem:partial_half-graph_has_rank=ladder-index} to extract a large biclique from a highly connected
partial half-graph.
This then yields a linear upper bound on the tree-width in terms of a forbidden semi-induced biclique:
\begin{theorem}\label{thm:sparse:tww2->bounded_tw}
Every graph of twin-width at most \(2\) which does not contain a \(K_{t,t}\)-subgraph has tree-width at most \(O(t)\).
\end{theorem}
Since the proof of \Cref{thm:sparse:tww2->bounded_tw} is pretty similar to that of \Cref{thm:stable:tww2->bounded_rw}
with some additional details borrowed from \cite{sparse_tww2_bounded_tw}, we move the proof to \Cref{appendix:sparse:tww2->bounded_tw}.

Further, the authors of \cite{sparse_tww2_bounded_tw} also showed that their result
does not generalize to graphs of twin-width \(3\), that is, there exists a family of sparse
graphs of twin-width \(3\) which has unbounded tree-width. Because sparse classes of bounded
rank-width also have bounded tree-width, this same family also has unbounded rank-width.
Thus, our theorem also does not generalize to graphs of twin-width \(3\).

Finally, using \Cref{lem:WL-dim_vs_rw}, we also get a bound on the Weisfeiler-Leman dimension.
\begin{corollary}
\(H_t\)-semi-free graphs of twin-width at most \(2\) are identified by the \(O(t)\)-dimensional Weisfeiler-Leman algorithm.
\end{corollary}

\section{Conclusion}
We considered the power of the Weisfeiler-Leman algorithm on graphs of small twin-width,
showing that graphs of twin-width \(1\) have Weisfeiler-Leman dimension at most \(3\),
while graphs of twin-width \(4\) have unbounded Weisfeiler-Leman dimension.
Moreover, we showed that stable graphs of twin-width \(2\) have bounded rank-width,
resolving a conjecture of~\cite{sparse_tww2_bounded_tw}, and also implying
a bound on the Weisfeiler-Leman dimension there.

It remains open if the Weisfeiler-Leman dimension
of graphs of twin-width \(1\) is 2 or 3,
and likewise, 
the precise Weisfeiler-Leman dimension
of stable graphs of twin-width 2 remains open.
Finally, it is not clear whether 
general graphs of twin-width \(2\) and \(3\)
have a bounded Weisfeiler-Leman dimension at all.

In our proof that 3-dimensional
Weisfeiler-Leman identifies graphs of twin-width 1,
we first showed the result restricted to prime graphs,
and then developed a general machinery that allows us
to lift identification results for prime graphs of a hereditary
graph class to all graphs in the class.
It remains to be seen if this technique extends to other classes
of graphs.

Recall the result of \cite{stable_bounded_tww} that every stable class of bounded twin-width is a transduction of a sparse class of bounded twin-width. Since transductions of classes of bounded tree-width
have bounded rank-width, a natural question is whether every stable class of twin-width \(2\) is a transduction of some sparse class of twin-width \(2\). If true, this would provide an alternative proof
of our result that stable graphs of twin-width \(2\) have bounded rank-width, albeit with worse bounds.
However, we do not know whether such a strengthening holds.

Furthermore, our techniques might extend to other width parameters where high connectivity
in partial half-graphs implies the existence of certain substructures.
A particularly interesting candidate might be mim-width \cite{mim-width}, which is similar to rank-width, with the rank function replaced by the size of a maximum semi-induced matching over a cut.
Since partial half-graphs contain no non-trivial semi-induced matchings, one might expect
that graphs of twin-width \(2\) have bounded mim-width. However, since the
connectivity function underlying mim-width is not submodular,
high mim-width does not imply the existence of large mim-well-linked sets,
which were a main ingredient in our proof.
Still, we conjecture that graphs of twin-width \(2\) have bounded mim-width.
Since permutation graphs have mim-width \(1\) and include all graphs of twin-width~\(1\) \cite{tww_one}, this is true for graphs of twin-width~\(1\).

Note that in general, twin-width and mim-width are incomparable, as exemplified by wall graphs and interval graphs \cite{comparing_widths}.

\bibliography{bibliography.bib}
\bibliographystyle{plainurl}

\begin{appendices}
\section{Proof of \texorpdfstring{\cref{tww4-logics}}{Lemma 3.3}}\label{appendix:interpretations}

We first introduce Logical Interpretations,
which are essentially the logical equivalent of algorithmic reductions.
For this, we first fix some notation.
\begin{definition}
A \emph{signature} $\tau$ is a tuple of symbols $(R_1,\ldots,R_k)$
with arities $(r_1,\ldots,r_k)$.
For a signature $\tau = (R_1,\ldots,R_k)$ with arities $(r_1,\ldots,r_k)$,
a \emph{$\tau$-structure} $\mathfrak{A}$
is a tuple $(A, R_1^{\mathfrak{A}},\ldots,R_k^{\mathfrak{A}})$
where $R_i^{\mathfrak{A}} \subseteq A^{r_i}$.
For a structure $\mathfrak{A}$ and a formula $\varphi(x_1,\ldots,x_r)$,
\[\varphi[\mathfrak{A}] \coloneqq
\{ \vec{a} \in A^r \mid {\mathfrak{A}} \vDash \varphi(\vec{a})\}.\]
\end{definition}

\noindent
In this setting, uncolored graphs can be seen as a structure with one
symmetric, anti-reflexive binary relation
that represents the edge set.
For colored graphs, it is possible to define colors with a finite number
of extra unary relations.
In this section, we will consider $C_k$-formulas that talk about graphs.
We will need the following terminology:

\begin{definition}[\cite{otto_2017}]
    A \emph{global relation} of arity $r$ is a mapping $R: \mathfrak{A} \mapsto R^{\mathfrak{A}}$
    where $R^{\mathfrak{A}} \subseteq A^r$ such that for all isomorphisms
    $\pi: \mathfrak{A} \rightarrow \mathfrak{B}$,
    it holds that $\pi(R^{\mathfrak{A}}) = R^{\mathfrak{B}}$.
    A formula $\varphi(x_1,\ldots,x_r)$ \emph{defines} a global relation $R$
    if for all structures $\mathfrak{A}$
    \[R^{\mathfrak{A}} = \varphi[\mathfrak{A}].\]
\end{definition}

\begin{definition}[Quotient structures]
	For a structure $\mathfrak{A} = (A,R_1,\ldots,R_k)$
	with relation arities $(r_1,\ldots,r_k$),
	and an equivalence relation ${\sim}$ on $A$,
	\begin{align*}
	\mathfrak{A} / {\sim} &\coloneqq
	(A/{\sim}, R_1/{\sim},\ldots,R_k/{\sim}) \text{ where } \\
	R_i/{\sim} &\coloneqq
	\{(a_1/{\sim},\ldots,a_{r_i}/{\sim}) \mid (a_1,\ldots,a_{r_i}) \in R_i\}.
	\end{align*}
\end{definition}

\begin{definition}[Restrictions]
	For a structure $\mathfrak{A} = (A,R_1,\ldots,R_k)$ and a set $D \subseteq A$,
	\[\mathfrak{A}|_D \coloneqq (D,R_1|_D,\ldots,R_k|_D).\]
\end{definition}

\noindent
A logical interpretation is a map on structures defined by a tuple of formulas:

\begin{definition}[Logical interpretations]
    A generalized $s$-dimensional \emph{$(\sigma,\tau)$-inter\-pre\-ta\-tion}
    is denoted as a tuple $\mathcal{I} = (\varphi_\text{dom}, \overline{\varphi}, \varphi_{\sim})$
    where the interpreted $\sigma$-structure over a $\tau$-structure ${\mathfrak{A}}$ is
    \[ \mathcal{I}(\mathfrak{A})
    = \big((A^s, \overline{\varphi}[\mathfrak{A}]) |_{\varphi_\text{dom}[\mathfrak{A}]} \big)
    / \varphi_{\sim}[\mathfrak{A}]\]
\end{definition}

\noindent
Most important for us is the following result:
\begin{lemma}[\cite{otto_2017}]\label{otto-lemma}
    For all $C_{sk}$-definable $s$-dimensional $(\sigma,\tau)$-interpretations $\mathcal{I}$
    and all $C_k$-definable global relations $R$ on the finite $\sigma$-structures,
    the global relation $\mathcal{I}(R)$ on the finite $\tau$-structures,
    where $\mathcal{I}(R)^{\mathfrak{A}} = R^{\mathcal{I}(\mathfrak{A})}$, is $C_{sk}$-definable
    in restriction to all those $\tau$-structures $\mathfrak{A}$ for which
    $\mathcal{I}(\mathfrak{A})$ is defined.
\end{lemma}

\noindent
We will make use of the following variant:

\begin{corollary}
    For all $C_{sk}$-definable $s$-dimensional $(\sigma,\tau)$-interpretations
    $\mathcal{I}$
    and all $C_k$-formulas $\varphi$ over $r$ variables,
    there is a $C_{sk}$-formula $\psi$ over $sr$ variables,
    such that for all $\mathfrak{A}$,
    \[\mathcal{I}(\mathfrak{A}),(\vec{x}_1,\ldots,\vec{x}_r) \vDash \varphi
    \text{ if and only if } \mathfrak{A},(x_1,\ldots,x_{sr}) \vDash \psi.\]
    We refer to this formula $\psi$ as $\mathcal{I}(\varphi)$.
\end{corollary}
\begin{proof}
    The formula $\varphi$ defines a global relation $R$.
    Via \cref{otto-lemma}, there is some $\psi$ that defines $\mathcal{I}(R)$.
\end{proof}

\noindent
In this way, it becomes clear that if we take a graph $G$
and are able to define a graph $H$ from $G$ via a $C_{sk}$-definable
$s$-dimensional interpretation,
and $H$ has Weisfeiler-Leman dimension $sk$,
then $G$ has Weisfeiler-Leman dimension $k$.
Now, we can prove the Lemma.

\begin{proof}
We define a 2{\nobreakdash-}dimensional interpretation
    $\mathcal{I} = (\varphi_\text{dom}, \varphi_E, \varphi_{\sim})$:
    
    \begin{align*}
        \varphi_\text{dom}(x,y) \coloneqq&\,Exy \vee x = y \\
        \varphi_E(x,y,x',y') \coloneqq&\,(x = y \wedge y = x' \wedge Ex'y') \\
                            \vee&\,(x = y \wedge y = y' \wedge Ex'y') \\
                            \vee&\,(x = x' \wedge x' = y' \wedge Exy) \\
                            \vee&\,(y = x' \wedge x' = y' \wedge Exy) \\
        \varphi_{\sim}(x,y,x',y') \coloneqq&\, x=y' \wedge y=x'
    \end{align*}
    
    \noindent
    In the interpreted structure on a graph $G$,
    	vertex pairs $(x,x)$ represent the vertices of the original
    	graph and vertex pairs $(x,y)$ represent a vertex that subdivides the edge 
    	from $x$ to $y$. 
    Because $\varphi_{\sim}$
    identifies the vertex pairs $(x,y)$ and $(y,x)$ that subdivide an edge,
    we have $\mathcal{I}(G) \cong G^1$ for all graphs $G$.
    Now, let $G,H$ be graphs such that $G \equiv_{C_{2k}} H$.
    Let $\varphi \in C_k$ be a formula with $r$ free variables.
    Then
    \begin{align*}
                \mathcal{I}(G) \vDash &\,  \varphi(\vec{x}_1,\ldots,\vec{x}_r) \\
        \text{if and only if }\hspace*{2mm} G
        \vDash &\, \mathcal{I}(\varphi)(\vec{x}_1\ldots\vec{x}_r) \\
        \text{if and only if }\hspace*{2mm} H
        \vDash &\, \mathcal{I}(\varphi)(\vec{x}_1\ldots\vec{x}_r) \\
        \text{if and only if }\hspace*{2mm} \mathcal{I}(H)
        \vDash &\, \varphi(\vec{x}_1,\ldots,\vec{x}_r)
    \end{align*}
    so $\mathcal{I}(G) \equiv_k \mathcal{I}(H)$.
    And therefore, since $G^1 \cong \mathcal{I}(G)$ and $H^1 \cong \mathcal{I}(H)$,
    also $G^1 \equiv_k H^1$.
\end{proof} 

\section{Proof of \texorpdfstring{\Cref{tww4-pebbles-restated}}{Theorem 3.4}}\label{tww4:pebbleproof}

In the subdivided graph, we will call vertices from the set
$V(G)$ \textit{original vertices},
and the newly added ones \textit{path vertices}.
In order to be able to refer to each individual
path vertex more clearly,
we fix a definition of ordered subdivisions.
\begin{definition}[$(s,<)$-subdivision]
    Let $G$ be a graph, $s \in \N$ and $<$ a strict total order on $V(G)$.
    The \emph{$(s,<)$-subdivision} of $G$ is
    \begin{align*}
         \subd{G} \coloneqq (&V(G) \cup \{s_{vwk} \mid vw \in E(G), v < w, 1 \leq k \leq s\}, \\
         &\{vs_{vw1} \mid vw \in E(G), v < w\} \\
         \cup \, &\{s_{vwk}s_{vw(k+1)} \mid vw \in E(G), v < w, 1 \leq k < s\} \\
         \cup \, &\{s_{vws}w \mid vw \in E(G), v < w\}).
    \end{align*}
\end{definition}

\noindent
It is clear that any $(s,<)$-subdivision
of a graph is also an $s$-subdivision.
Therefore, the Lemma can be restated and proved as follows.

\begin{lemma} \label{tww4-pebbles}
Let $s \in \N$, and $G,H$ be graphs
with $V(G) = V(H)$.
Let $<$ be some order on $V(G)$.
Then for all $k$, if $2k$-WL does not distinguish $G^{(3,<)}$
and $H^{(3,<)}$,
then $k$-WL does not distinguish $\subd{G}$ and $\subd{H}$.
\end{lemma}

\begin{proof}
	We prove the Lemma
	by extending the strategy played by the Duplicator
	in the pebble game
	$\bp{2k}(G^{(3,<)},H^{(3,<)})$
	to the pebble game $\bpk(\subd{G},\subd{H})$.
	We will refer to the Duplicator in the former game
	as $D$, and the Duplicator in the latter game as $D^s$.
    Similarly, we refer to Spoilers as $S$ and $S^s$.
    
    First, we will describe the invariant that $D^s$ maintains and how she
    plays each round
    by simulating a game in the 3-subdivided graphs and querying $D$
	for a response.
    Then, we prove that each move by $D^s$ is valid
    and maintains the invariant, and finally, that by doing so she wins the game.
    
    There are two positions at the start of each round:
    The current position $D^s$ is responding to, which we will call $p^s$,
    and the simulated position that $D$ is responding to,
    which we call $p$.
    In the following, we use the auxiliary function
    \[ d(k) \coloneqq \begin{cases}
        1 \text{ if } k \leq \frac{s}{2},  \\
        3 \text{ otherwise.}
    \end{cases}\]
    This will indicate the \enquote{direction} the Duplicators are playing on each edge.
    $D^s$ will uphold the following invariant
    regarding the positions $p$ and $p^s$:
    \begin{itemize}
        \item The position $p$ is a winning position for $D$.
        \item All pairs $(v,v')$ appearing in $p^s$ are either both original vertices
        or both path vertices.
        \item All pairs $(v,v')$ of original vertices appearing in $p^s$ also appear in $p$.
        \item For all pairs $(s_{ek},s_{e'k'})$ of path vertices that appear in $p^s$,
        the pairs $(s_{e2},s_{e'2})$ and $(s_{ed(k)},s_{e'd(k')})$
        appear in $p$.
    \end{itemize}
    \noindent
    To maintain the invariant, $D^s$ plays each round as follows:
    \begin{enumerate}
        \item Let $b$ be the bijection chosen by $D$ for the position $p$.
        Then $D^s$ chooses the bijection $b^s$ defined as follows:
        \begin{enumerate}
        		\item Let $(v,w)$ be the pebble pair in the position $p^s$ picked up by $S^s$,
        		if pebbles were picked up in this round.
        		Say $(v,w)$ is the pair at index~$i$.
        		Then pick up the pebble pairs in position $p$ at indices $2i$ and $2i+1$.
            \item For original vertices $v \in V(G)$, define $b^s(v) = b(v)$.
            \item For path vertices $s_{ek}$, let $p'$ be the position $p$
            with the added pebble pair $(s_{e2},b(s_{e2}))$.
            Let $b'$ be the bijection that $D$
            picks for the position $p'$.
            Then
            \[ b^s(s_{ek}) = 
            \begin{cases}
                s_{e'k} \text{ if } b'(s_{e1}) = s_{e'1}
                \text{ for some } e' \in E(H) \\
                s_{e'(s-k+1)} \text{ if } b'(s_{e1}) = s_{e'3}
                \text{ for some } e' \in E(H).
            \end{cases}\]
        \end{enumerate}
        \item Let $(v,b^s(v))$ be the response pebble pair of $S^s$.
        Let $i$ be the index of this pebble pair.
        Then if $v$ is an original vertex,
        add the pebble pair $(v,b(v))$ to the position $p$
        at index $2i$.
        If it is a path vertex, say $v = s_{ek}$ and $b^s(v) = s_{e'k'}$,
        place $(s_{e2},s_{e'2})$ at index $2i$
        and $(s_{ed(k)},s_{e'd(k')})$ at index $2i+1$.
    \end{enumerate}
    Intuitively, $D^s$ copies the bijection of $D$ on the original vertices
    and then maps the path vertices by looking ahead one move and seeing the
    \enquote{direction} that $D$ will map the path in and extending that map. This is a winning strategy for $D^s$.

First, note that $b^s$ is always well-defined:
    Duplicator
    $D$ will always chose a bijection that maps original to original
    vertices and path to path vertices, because the original vertices have degree
    at least 3 according to our assumptions on $G$ and~$H$,
    and path vertices have degree 2.
    Next, $D$ also always has to map a \enquote{middle} path vertex
    $s_{e2}$ to some $s_{e'2}$ for another edge $e'$,
    as otherwise Spoiler can win in the next round by pebbling $s_{e2}$
    and then $s_{e1}$ or $s_{e3}$. See \Cref{fig:middlepath}.
    
    \begin{figure}
	\centering
		\begin{tikzpicture}[minimum size = 1em,
    			baseline=(current bounding box.north),
    				every path/.style={thick}]
		\begin{scope}[every node/.style={circle,draw}]
    			\node[minimum size=1.5em] (v) at (0,5) {};
   			\node[label=left:{$s_{e1}$}] (e1) at (0,4) {};
    			\node[label=left:{$s_{e2}$}] (e2) at (0,3) {};
    			\node[label=left:{$s_{e3}$}] (e3) at (0,2) {};
    			\node[minimum size=1.5em] (w) at (0,1) {};
    			\node[minimum size=1.5em] (v') at (3,5) {};
   			\node[label=right:{$s_{e'1}$}] (e'1) at (3,4) {};
    			\node[label=right:{$s_{e'2}$}] (e'2) at (3,3) {};
    			\node[label=right:{$s_{e'3}$}] (e'3) at (3,2) {};
    			\node[minimum size=1.5em] (w') at (3,1) {};
		\end{scope}
		\draw (v) -- (e1);
		\draw (e1) -- (e2);
		\draw (e2) -- (e3);
		\draw (e3) -- (w);
		\draw (v') -- (e'1);
		\draw (e'1) -- (e'2);
		\draw (e'2) -- (e'3);
		\draw (e'3) -- (w');
		\path[->, draw=blue,] (e2) edge (e'1);
		\node (check) at (1.5,0) {{\color{red}\ding{55}}};
		\end{tikzpicture}
		\hspace{5em}
		\begin{tikzpicture}[minimum size = 1em,
    			baseline=(current bounding box.north),
    				every path/.style={thick}]
		\begin{scope}[every node/.style={circle,draw}]
    			\node[minimum size=1.5em] (v) at (0,5) {};
   			\node[label=left:{$s_{e1}$}] (e1) at (0,4) {};
    			\node[label=left:{$s_{e2}$}] (e2) at (0,3) {};
    			\node[label=left:{$s_{e3}$}] (e3) at (0,2) {};
    			\node[minimum size=1.5em] (w) at (0,1) {};
    			\node[minimum size=1.5em] (v') at (3,5) {};
   			\node[label=right:{$s_{e'1}$}] (e'1) at (3,4) {};
    			\node[label=right:{$s_{e'2}$}] (e'2) at (3,3) {};
    			\node[label=right:{$s_{e'3}$}] (e'3) at (3,2) {};
    			\node[minimum size=1.5em] (w') at (3,1) {};
		\end{scope}
		\draw (v) -- (e1);
		\draw (e1) -- (e2);
		\draw (e2) -- (e3);5
		\draw (e3) -- (w);
		\draw (v') -- (e'1);
		\draw (e'1) -- (e'2);
		\draw (e'2) -- (e'3);
		\draw (e'3) -- (w');
		\path[->, draw=blue] (e2) edge (e'2);
		\node (check) at (1.5,0) {{\color{Green}\ding{51}}};
		\end{tikzpicture}
    	\caption{$D$ has to map middle path vertices to middle path vertices.
    	The bijection chosen by $D$ is indicated by the arrows.}
    	\label{fig:middlepath}
    \end{figure}
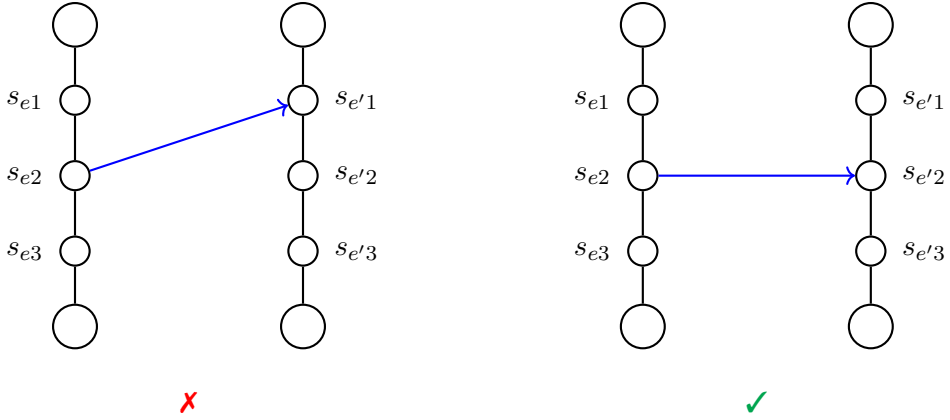
    
    Finally, it is a bijection, because $b$ on the original vertices
    is a bijection, and $b$ on the \enquote{middle} path vertices is a bijection,
    as there must be the same amount of edges in both graphs - 
    otherwise, Spoiler wins $\bpk(G^{(3,<)},H^{(3,<)})$
    by pebbling a vertex $(v,b(v))$ pair of differing degrees 
    and in the next move pebbling the vertex that is a neighbor
    of one of the two but
    not the other.
    And after pebbling some pair $(s_{e2},s_{e'2})$,
    $D$ has to map $\{s_{e_1},s_{e_3}\}$ to $\{s_{e'1},s_{e'3}\}$
    in the next move, otherwise Spoiler wins right away.
    Next, note that $D$ maintains a winning position because
    she has a winning strategy for the starting position,
    and the pebbles added to the position $p$ represent legal moves of Spoiler
    in the simulated game.
    The rest of the invariant is fulfilled by construction.
    Therefore, by initializing $p$ and $p^s$ as the empty position
    and then using this protocol for every round, $D^s$ maintains the above invariant.
    
    Next, we show that $D^s$ follows a winning strategy.
    It follows from the definition of $b^s$ that the bijection always
    respects existing pebbles:
    Once a path vertex on the path corresponding to some edge $e$ has been pebbled,
    $b^s$ is \enquote{fixed} for that pebble, as both $s_{e2}$ and either $s_{e1}$
    or $s_{e3}$ have been pebbled, determining what bijection $D$ has to pick
    on that path.
    Now, assume that in some round Duplicator has picked the bijection $b^s$
    in position $p^s = (\vec{v},b^s(\vec{v}))$
    and Spoiler has picked a new vertex pair
    $(v,v')$ such that $\atp_{\subd{G}}(\vec{v}v) \neq \atp_{\subd{H}}(b^s(\vec{v})v')$.
    Let $w$ be the vertex in $\vec{v}$ that \enquote{disagrees},
    that is $vw \in E(\subd{G})$ while
    $b^s(v)b^s(w) \not \in E(\subd{H})$, or the other way around.
    By symmetry, let $v,w$ be the adjacent pair.
    We make a case distinction.
    
    \begin{itemize}
        \item \emph{Case 1}:
        Both $v$ and $w$ are original vertices.
        Then $b^s(v)$ and $b^s(w)$ are also original vertices.
        Therefore, neither $v,w$ nor $b^s(v),b^s(w)$
        are adjacent in either of the subdivided graphs
        as they are connected by a path. See \Cref{fig:pebblecase1}.
		 \begin{figure}
			\centering
			\begin{tikzpicture}[minimum size = 1em,
    				baseline=(current bounding box.north),
    				every path/.style={thick}]
    			\node[draw=none] (Gsl) at (0,6) {$G^{(s,<)}$};
    			\node[draw=none] (Hsl) at (3,6) {$H^{(s,<)}$};
			\begin{scope}[every node/.style={circle,draw}]
    			\node[minimum size=1.5em] (v) at (0,5) {};
   			\node (e1) at (0,4) {};
    			\node (e3) at (0,2) {};
    			\node[minimum size=1.5em] (w) at (0,1) {};
    			\node[minimum size=1.5em] (v') at (3,5) {};
   			\node (e'1) at (3,4) {};
    			\node (e'3) at (3,2) {};
    			\node[minimum size=1.5em] (w') at (3,1) {};
			\end{scope}
			\draw (v) -- (e1);
			\draw[dashed] (e1) -- (e3);
			\draw (e3) -- (w);
			\draw (v') -- (e'1);
			\draw[dashed] (e'1) -- (e'3);
			\draw (e'3) -- (w');
			\path[->, draw=blue,] (v) edge (v');
			\path[->, draw=blue,] (w) edge (w');
		\end{tikzpicture}
		\hspace{5em}
		\begin{tikzpicture}[minimum size = 1em,
    			baseline=(current bounding box.north),
    				every path/.style={thick}]
    			\node[draw=none] (G3l) at (0,6) {$G^3$};
    			\node[draw=none] (H3l) at (3,6) {$H^3$};
			\begin{scope}[every node/.style={circle,draw}]
    			\node[minimum size=1.5em] (v) at (0,5) {};
   			\node (e1) at (0,4) {};
    			\node (e2) at (0,3) {};
    			\node (e3) at (0,2) {};
    			\node[minimum size=1.5em] (w) at (0,1) {};
    			\node[minimum size=1.5em] (v') at (3,5) {};
   			\node (e'1) at (3,4) {};
    			\node (e'2) at (3,3) {};
    			\node (e'3) at (3,2) {};
    			\node[minimum size=1.5em] (w') at (3,1) {};
			\end{scope}
			\draw (v) -- (e1);
			\draw (e1) -- (e2);
			\draw (e2) -- (e3);5
			\draw (e3) -- (w);
			\draw (v') -- (e'1);
			\draw (e'1) -- (e'2);
			\draw (e'2) -- (e'3);
			\draw (e'3) -- (w');
			\path[->, draw=blue] (v) edge (v');
			\path[->, draw=blue] (w) edge (w');
		\end{tikzpicture}
    	\caption{Illustration of Case 1 of the proof of \Cref{tww4-pebbles}.
    	The bijections chosen by $D$ and $D^s$ are indicated by the arrows.}
    	\label{fig:pebblecase1}
    \end{figure}
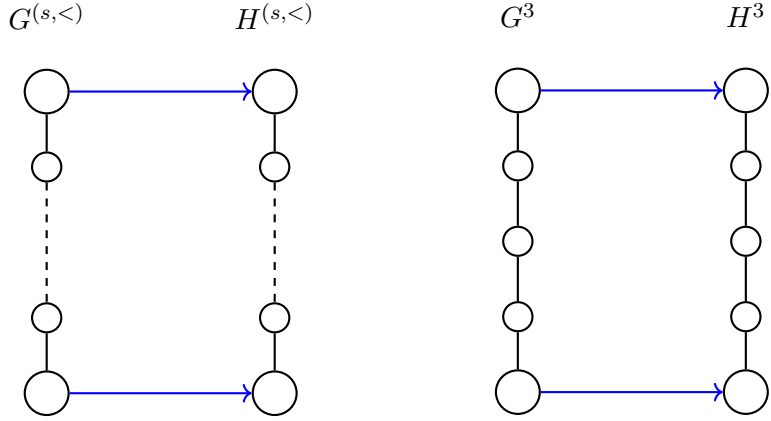
        \item \emph{Case 2}:
        Both $v$ and $w$ are path vertices.
        Due to how $b^s$ is defined,
        $v$ and $w$ are adjacent if and only if $b^s(v)$ and $b^s(w)$ are.
        See \Cref{fig:pebblecase2}.
        \begin{figure}
			\centering
			\begin{tikzpicture}[minimum size = 1em,
    				baseline=(current bounding box.north),
    				every path/.style={thick}]
    			\node[draw=none] (Gsl) at (0,6) {$G^{(s,<)}$};
    			\node[draw=none] (Hsl) at (3,6) {$H^{(s,<)}$};
			\begin{scope}[every node/.style={circle,draw}]
    			\node[minimum size=1.5em] (v) at (0,5) {};
   			\node (e1) at (0,4) {};
    			\node (e3) at (0,2) {};
    			\node[minimum size=1.5em] (w) at (0,1) {};
    			\node[minimum size=1.5em] (v') at (3,5) {};
   			\node (e'1) at (3,4) {};
    			\node (e'3) at (3,2) {};
    			\node[minimum size=1.5em] (w') at (3,1) {};
			\end{scope}
			\draw (v) -- (e1);
			\draw[dashed] (e1) -- (e3);
			\draw (e3) -- (w);
			\draw (v') -- (e'1);
			\draw[dashed] (e'1) -- (e'3);
			\draw (e'3) -- (w');
			\path[->, draw=blue] (e1) edge (e'1);
			\path[->, draw=blue] (e3) edge (e'3);
		\end{tikzpicture}
		\hspace{5em}
		\begin{tikzpicture}[minimum size = 1em,
    			baseline=(current bounding box.north),
    				every path/.style={thick}]
    			\node[draw=none] (G3l) at (0,6) {$G^3$};
    			\node[draw=none] (H3l) at (3,6) {$H^3$};
			\begin{scope}[every node/.style={circle,draw}]
    			\node[minimum size=1.5em] (v) at (0,5) {};
   			\node (e1) at (0,4) {};
    			\node (e2) at (0,3) {};
    			\node (e3) at (0,2) {};
    			\node[minimum size=1.5em] (w) at (0,1) {};
    			\node[minimum size=1.5em] (v') at (3,5) {};
   			\node (e'1) at (3,4) {};
    			\node (e'2) at (3,3) {};
    			\node (e'3) at (3,2) {};
    			\node[minimum size=1.5em] (w') at (3,1) {};
			\end{scope}
			\draw (v) -- (e1);
			\draw (e1) -- (e2);
			\draw (e2) -- (e3);5
			\draw (e3) -- (w);
			\draw (v') -- (e'1);
			\draw (e'1) -- (e'2);
			\draw (e'2) -- (e'3);
			\draw (e'3) -- (w');
			\path[->, draw=blue] (e1) edge (e'1);
			\path[->, draw=blue] (e2) edge (e'2);
			\path[->, draw=blue] (e3) edge (e'3);
		\end{tikzpicture}
    	\caption{Illustration of Case 2 of the proof of \Cref{tww4-pebbles}.}
    	\label{fig:pebblecase2}
    \end{figure}
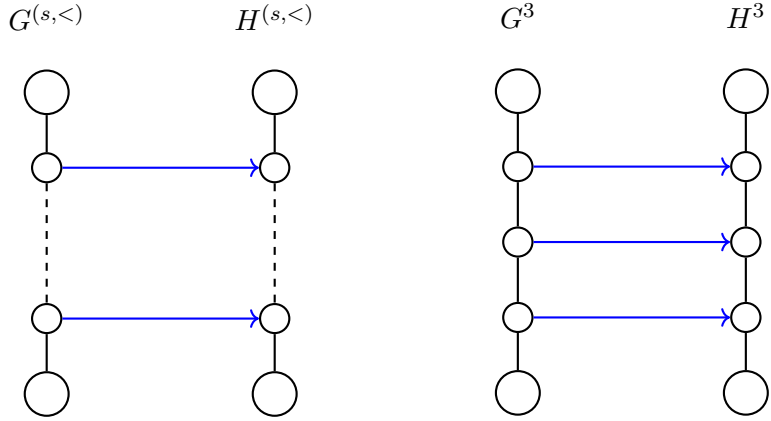
        \item \emph{Case 3}: The vertex $v$ is a path vertex and $w$ is an original vertex.
        Then $v$ is the first vertex on a path from $w$ to some original vertex $x$.
        By symmetry, assume $w<x$ and therefore $v = s_{wx1}$.
        Otherwise the proof runs analogously but with $v = s_{wxs}$.
        Now $b(s_{wx2}) = s_{yz2}$ for some $y,z \in V(H)$.
        It holds that either $b(w) = y$ or $b(w) = z$
        as otherwise, Spoiler can first pebble $s_{wx2}$ and then win
        because $s_{wx1}$ is adjacent to both $w$ and $s_{wx2}$,
        $(w,b(w))$ is pebbled in this position, and there is no vertex in $H^3$
        adjacent to both $b(w)$ and $s_{yz2}$ if $b(w)$ is neither $y$ nor $z$.
        So without loss of generality, let $b(w) = y$.
        Now, in the game in which $(s_{wx2},s_{yz2})$ is pebbled,
        Duplicator has to send $s_{wx1}$ to the vertex between $b(w)$ and $s_{b(w)z2}$,
        and therefore, $b^s$ maps $v=s_{wx1}$ to the first vertex on the path from
        $b^s(w)=b(w)$ to $z$, meaning $b^s(v)$ and $b^s(w)$ are indeed adjacent.
        See \Cref{fig:pebblecase3}.
		
		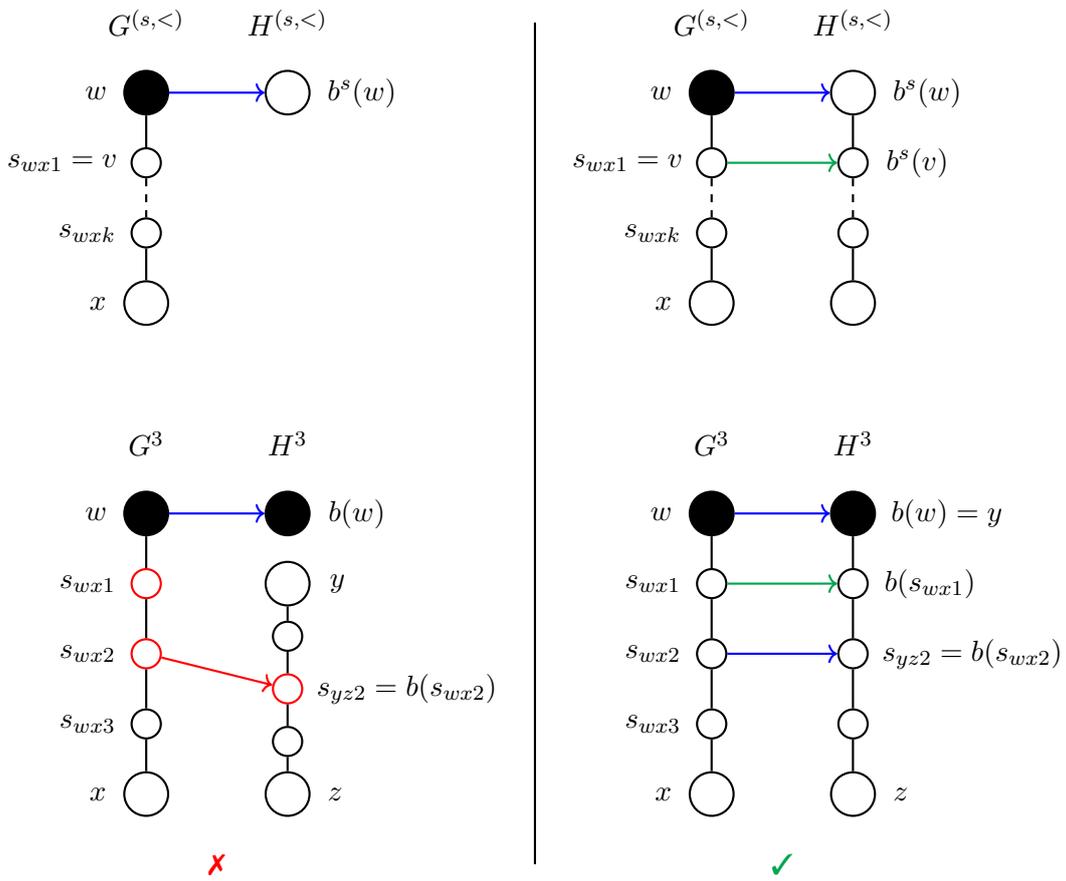
\begin{figure}
		\centering
			\begin{tikzpicture}[minimum size = 1em,
    				baseline=(current bounding box.north),
    				every path/.style={thick}, scale=0.93]
    			//top left
    			\node (Gsl) at (0,4) {$G^{(s,<)}$};
    			\node (Hsl) at (2,4) {$H^{(s,<)}$};
			\begin{scope}[every node/.style={circle,draw}]
    			\node[label=left:{$w$},fill=black,minimum size = 1.5em]
    			 	(wl1) at (0,3) {};
    			\node[label=left:{$s_{wx1}=v$}]
    				(vl) at (0,2) {};
    			\node[label=left:{$s_{wxk}$}]
    				(swxkl) at (0,1) {};
    			\node[label=left:{$x$},minimum size = 1.5em]
    				(xl1) at (0,0) {};
    			\node[label=right:{$b^s(w)$},minimum size = 1.5em]
    				(bswl) at (2,3) {};
			\end{scope}
			\draw (wl1) -- (vl);
			\draw[dashed] (vl) -- (swxkl);
			\draw (swxkl) -- (xl1);
			\path[->,draw=blue] (wl1) edge (bswl);
			
			//bottom left
    			\node (G3l) at (0,-2) {$G^3$};
    			\node (G3l) at (2,-2) {$H^3$};
			\begin{scope}[every node/.style={circle,draw}]
    			\node[label=left:{$w$},fill=black,minimum size = 1.5em]
    			 	(wl2) at (0,-3) {};
    			\node[label=left:{$s_{wx1}$},draw=red]
    				(s1l) at (0,-4) {};
    			\node[label=left:{$s_{wx2}$},draw=red]
    				(s2l) at (0,-5) {};
    			\node[label=left:{$s_{wx3}$}]
    				(s3l) at (0,-6) {};
    			\node[label=left:{$x$},minimum size = 1.5em]
    			 	(xl2) at (0,-7) {};
    			\node[label=right:{$b(w)$},fill=black,minimum size = 1.5em]
    			 	(bwl) at (2,-3) {};
    			\node[label=right:{$y$},minimum size = 1.5em]
    				(yl) at (2,-4) {};
    			\node (syz1l) at (2,-4.75) {};
    			\node[label=right:{$s_{yz2} = b(s_{wx2})$}, draw=red]
    				(syz2l) at (2,-5.5) {};
    			\node (syz3l) at (2,-6.25) {};
    			\node[label=right:{$z$},minimum size = 1.5em]
    				(zl) at (2,-7) {};
			\end{scope}
			\draw (wl2) -- (s1l);
			\draw (s1l) -- (s2l);
			\draw (s2l) -- (s3l);
			\draw (s3l) -- (xl2);
			\draw (yl) -- (syz1l);
			\draw (syz1l) -- (syz2l);
			\draw (syz2l) -- (syz3l);
			\draw (syz3l) -- (zl);
			\path[->,draw=blue] (wl2) edge (bwl);
			\path[->,draw=red] (s2l) edge (syz2l);
			\node (check) at (1,-8) {{\color{red}\ding{55}}};
			
			\draw (5.5,4) -- (5.5,-8);
			
			//top right
    			\node (Gsr) at (8,4) {$G^{(s,<)}$};
    			\node (Hsr) at (10,4) {$H^{(s,<)}$};
			\begin{scope}[every node/.style={circle,draw}]
    			\node[label=left:{$w$},fill=black,minimum size = 1.5em]
    			 	(wr1) at (8,3) {};
    			\node[label=left:{$s_{wx1}=v$}]
    				(vr) at (8,2) {};
    			\node[label=left:{$s_{wxk}$}]
    				(swxkr) at (8,1) {};
    			\node[label=left:{$x$},minimum size = 1.5em]
    				(xr1) at (8,0) {};
    			\node[label=right:{$b^s(w)$},minimum size = 1.5em]
    				(bswr) at (10,3) {};
    			\node[label=right:{$b^s(v)$}]
    				(bsv) at (10,2) {};
    			\node(yendr) at (10,1) {};
    			\node[minimum size = 1.5em]
    			 	(endr) at (10,0) {};
			\end{scope}
			\draw (wr1) -- (vr);
			\draw[dashed] (vr) -- (swxkr);
			\draw (swxkr) -- (xr1);
			\draw (bswr) -- (bsv);
    			\draw[dashed] (bsv) -- (yendr);
			\draw (yendr) -- (endr);
			\path[->,draw=blue] (wr1) edge (bswr);
			\path[->,draw=Green] (vr) edge (bsv);
			
			//bottom right
    			\node (G3r) at (8,-2) {$G^3$};
    			\node (G3r) at (10,-2) {$H^3$};
			\begin{scope}[every node/.style={circle,draw}]
    			\node[label=left:{$w$},fill=black,minimum size = 1.5em]
    			 	(wr2) at (8,-3) {};
    			\node[label=left:{$s_{wx1}$}]
    				(s1r) at (8,-4) {};
    			\node[label=left:{$s_{wx2}$}]
    				(s2r) at (8,-5) {};
    			\node[label=left:{$s_{wx3}$}]
    				(s3r) at (8,-6) {};
    			\node[label=left:{$x$},minimum size = 1.5em]
    			 	(xr2) at (8,-7) {};
    			\node[label=right:{$b(w) = y$},fill=black,minimum size = 1.5em]
    			 	(bwr) at (10,-3) {};
    			\node[label=right:{$b(s_{wx1})$}] (syz1r) at (10,-4) {};
    			\node[label=right:{$s_{yz2} = b(s_{wx2})$}]
    				(syz2r) at (10,-5) {};
    			\node (syz3r) at (10,-6) {};
    			\node[label=right:{$z$},minimum size = 1.5em]
    				(zr) at (10,-7) {};
			\end{scope}
			\draw (wr2) -- (s1r);
			\draw (s1r) -- (s2r);
			\draw (s2r) -- (s3r);
			\draw (s3r) -- (xr2);
			\draw (bwr) -- (syz1r);
			\draw (syz1r) -- (syz2r);
			\draw (syz2r) -- (syz3r);
			\draw (syz3r) -- (zr);
			\path[->,draw=blue] (wr2) edge (bwr);
			\path[->,draw=blue] (s2r) edge (syz2r);
			\path[->,draw=Green] (s1r) edge (syz1r);
			\node (cross) at (9,-8) {{\color{Green}\ding{51}}};
			\end{tikzpicture}  
    		\caption{Illustration of Case 3 of the proof of \Cref{tww4-pebbles}.
    		Pebbled vertices are drawn in black.}
    		\label{fig:pebblecase3}
		\end{figure}    
        
        \item \emph{Case 4}: The vertex $v$ is an original vertex, and $w$ is a path vertex.
        Then $w$ is the first vertex on a path from $v$ to some original vertex
        $x$. By symmetry $v<x$ and therefore $w=s_{vx1}$,
        otherwise $w=s_{vxs}$ and the proof runs analogously.
        We know due to how $b^s$ is defined that for some $y,z \in V(H)$,
        either $b^s(s_{vx1})=s_{yz1}$ or $b^s(s_{vx1})=s_{yzs}$.
        \begin{itemize}
            \item \emph{Case 4.1:} Assume $b^s(s_{vx1})=s_{yz1}$.
            Due to the invariant $D^s$ maintains, $(s_{vx1},s_{yz1})$
            and $(s_{vx2},s_{yz2})$ are pebbled in $\bp{2k}(G^{(3,<)},H^{(3,<)})$,
            and therefore $b(v)=y$ must hold in this round as otherwise Spoiler wins
            right away. Therefore $b^s(v)=y$ and so $b^s(v)=y$ and $b^s(w)=s_{yz1}$
            are adjacent.
            See the left half of \Cref{fig:pebblecase3}.
            \item \emph{Case 4.2:} Otherwise, $b^s(s_{vx1})=s_{yzs}$.
            Analogous argument, only now $(s_{vx2},s_{yz2})$ and $(s_{vx1},s_{yz3})$
            are pebbled in $\bp{2k}(G^{(3,<)},H^{(3,<)})$ and therefore $b(v) = z$ must hold.
            See the right half of \Cref{fig:pebblecase3}.
        \end{itemize}
        
        \begin{figure}
        \centering
			\begin{tikzpicture}[minimum size = 1em,
    				baseline=(current bounding box.north),
    				every path/.style={thick},
    				scale=0.97]
    			//top left
    			\node (Gsl) at (0,4) {$G^{(s,<)}$};
    			\node (Hsl) at (2,4) {$H^{(s,<)}$};
			\begin{scope}[every node/.style={circle,draw}]
    			\node[label=left:{$v$},minimum size = 1.5em]
    			 	(vl1) at (0,3) {};
    			\node[label=left:{$s_{vx1}=w$},fill=black]
    				(wl) at (0,2) {};
    			\node (svxkl) at (0,1) {};
    			\node[label=left:{$x$},minimum size = 1.5em]
    				(xl1) at (0,0) {};
    			\node[label=right:{$y = b(v)$},minimum size = 1.5em]
    				(bv) at (2,3) {};
    			\node[label=right:{$b^s(w) = s_{yz1}$},fill=black]
    				(bswl) at (2,2) {};
    			\node(syzk1) at (2,1) {};
    			\node[label=right:{$z$},minimum size = 1.5em]
    				(zl1) at (2,0) {};
			\end{scope}
			\draw (vl1) -- (wl);
			\draw[dashed] (wl) -- (svxkl);
			\draw (svxkl) -- (xl1);
			\path[->,draw=Green] (vl1) edge (bv);
			\path[->,draw=blue] (wl) edge (bswl);
			\draw (bv) -- (bswl);
			\draw[dashed] (bswl) -- (syzk1);
			\draw (syzk1) -- (zl1);
			
			//bottom left
    			\node (G3l) at (0,-2) {$G^3$};
    			\node (G3l) at (2,-2) {$H^3$};
			\begin{scope}[every node/.style={circle,draw}]
    			\node[label=left:{$v$},minimum size = 1.5em]
    			 	(vl2) at (0,-3) {};
    			\node[label=left:{$s_{wv1}$},fill=black]
    				(s1l) at (0,-4) {};
    			\node[label=left:{$s_{vx2}$},fill=black]
    				(s2l) at (0,-5) {};
    			\node(s3l) at (0,-6) {};
    			\node[label=left:{$x$},minimum size = 1.5em]
    			 	(xl2) at (0,-7) {};
    			\node[label=right:{y},minimum size = 1.5em]
    			 	(yl) at (2,-3) {};
    			\node[label=right:{$s_{yz1}$},fill=black] (syz1l) at (2,-4) {};
    			\node[label=right:{$s_{yz2}$}, fill=black]
    				(syz2l) at (2,-5) {};
    			\node (syz3l) at (2,-6) {};
    			\node[label=right:{$z$},minimum size = 1.5em]
    				(zl2) at (2,-7) {};
			\end{scope}
			\draw (vl2) -- (s1l);
			\draw (s1l) -- (s2l);
			\draw (s2l) -- (s3l);
			\draw (s3l) -- (xl2);
			\draw (yl) -- (syz1l);
			\draw (syz1l) -- (syz2l);
			\draw (syz2l) -- (syz3l);
			\draw (syz3l) -- (zl2);
			\path[->,draw=Green] (wl2) edge (yl);
			\path[->,draw=blue] (s1l) edge (syz1l);
			\path[->,draw=blue] (s2l) edge (syz2l);
			
			\draw (5.5,4) -- (5.5,-7);
			
			//top right
    			\node (Gsr) at (8,4) {$G^{(s,<)}$};
    			\node (Hsr) at (10,4) {$H^{(s,<)}$};
			\begin{scope}[every node/.style={circle,draw}]
    			\node[label=left:{$v$},minimum size = 1.5em]
    			 	(vr1) at (8,3) {};
    			\node[label=left:{$s_{vx1}=w$},fill=black]
    				(wr) at (8,2) {};
    			\node (svxkr) at (8,1) {};
    			\node[label=left:{$x$},minimum size = 1.5em]
    				(xr1) at (8,0) {};
    			\node[label=right:{$y$},minimum size = 1.5em]
    				(yr1) at (10,3) {};
    			\node
    				(bswr) at (10,2) {};
    			\node[label=right:{$b^s(w) = s_{yzs}$},fill=black]
    				(syzk1) at (10,1) {};
    			\node[label=right:{$z = b^s(v)$},minimum size = 1.5em]
    				(zr1) at (10,0) {};
			\end{scope}
			\draw (vr1) -- (wr);
			\draw[dashed] (wr) -- (svxkr);
			\draw (svxkr) -- (xr1);
			\path[->,draw=Green] (vr1) edge (zr1);
			\path[->,draw=blue] (wr) edge (syzk1);
			\draw (yr1) -- (bswr);
			\draw[dashed] (bswr) -- (syzk1);
			\draw (syzk1) -- (zr1);
			
			//bottom right
    			
    			\node (G3r) at (8,-2) {$G^3$};
    			\node (G3r) at (10,-2) {$H^3$};
			\begin{scope}[every node/.style={circle,draw}]
    			\node[label=left:{$v$},minimum size = 1.5em]
    			 	(vr2) at (8,-3) {};
    			\node[label=left:{$s_{wv1}$},fill=black]
    				(s1r) at (8,-4) {};
    			\node[label=left:{$s_{vx2}$},fill=black]
    				(s2r) at (8,-5) {};
    			\node(s3r) at (8,-6) {};
    			\node[label=left:{$x$},minimum size = 1.5em]
    			 	(xr2) at (8,-7) {};
    			\node[label=right:{y},minimum size = 1.5em]
    			 	(yr) at (10,-3) {};
    			\node (syz1r) at (10,-4) {};
    			\node[label=right:{$s_{yz2}$}, fill=black]
    				(syz2r) at (10,-5) {};
    			\node[label=right:{$s_{yz3}$},fill=black]
    				(syz3r) at (10,-6) {};
    			\node[label=right:{$z$},minimum size = 1.5em]
    				(zr2) at (10,-7) {};
			\end{scope}
			\draw (vr2) -- (s1r);
			\draw (s1r) -- (s2r);
			\draw (s2r) -- (s3r);
			\draw (s3r) -- (xr2);
			\draw (yr) -- (syz1r);
			\draw (syz1r) -- (syz2r);
			\draw (syz2r) -- (syz3r);
			\draw (syz3r) -- (zr2);
			\path[->,draw=Green] (wr2) edge (zr2);
			\path[->,draw=blue] (s1r) edge (syz3r);
			\path[->,draw=blue] (s2r) edge (syz2r);
			
			\draw (5.5,4) -- (5.5,-7);
			\end{tikzpicture}
    		\caption{Illustration of Case 4 of the proof of \Cref{tww4-pebbles}.}
    		\label{fig:pebblecase4}
        \end{figure}
    \end{itemize}
    In this way, $D^s$ wins $\bpk(\subd{G},\subd{H})$. 
\end{proof}
\section{Proof of \texorpdfstring{\cref{csguv}}{Lemma 4.4}}\label{cssequence}
\begin{proof}
For the forward direction,
let $f\colon G^c \cong H^c$ be an isomorphism
and $u,v \in V(G)$ such that $\cs(G^c,u,v) \neq \text{f}$,
then it is clear to see that $\cs(G^c,u,v) = \cs(H^c,f(u),f(v))$.
For the backward direction,
one can define an isomorphism $f \colon G^c \cong H^c$
by following the canonical contraction sequence laid out in $\cs(G^c,u,v)$.
Consider the series of graphs $G_i$, $H_i$
and vertices $u_{G,i},v_{G,i},u_{H,i},v_{H,i}$
as described in the computation of $\text{cs}(G,u,v)$
and $\text{cs}(H,u',v')$.
Then the isomorphism can be defined as follows:
\begin{itemize}
\item Set $f(u) = u'$ and $f(v) = v'$.
\item For the first iteration $i = 0$ as well as all iterations $i$
in which a red edge gets contracted, set $f(v_{G,i+1}) = v_{H,i+1}$.
\item For every other iteration $i$,
define $f$ in such a way that for all $T$ and $c$, we have
\[f(\ntw_{G_i}^{T,u_{G,i},c}(u_{G,i},v_{G,i})) =
\ntw_{H_i}^{T,u_{H,i},c}(u_{H,i},v_{H,i})\]
and
\[f(\ntw_{G_i}^{T,v_{G,i},c}(u_{G,i},v_{G,i})) =
\ntw_{H_i}^{T,v_{H,i},c}(u_{H,i},v_{H,i})\]
in any bijective way that respects the colors.
This is the case when two vertices were part of the same contraction phase,
and therefore well defined because there was always the same amount of
near twins respective each color and pair of atomic types in each iteration.
\end{itemize}
\noindent
To prove that this is indeed an isomorphism, first note that $f$ respects
the coloring of the graphs. Assume $f$ is not an isomorphism.
We show that if there are ${w_1,w_2 \in V(G)}$ such that
$w_1w_2 \in E(G)$ but $f(w_1)f(w_2) \not \in E(H)$,
we derive a contradiction.
For $w_1w_2 \not \in E(G)$ but $f(w_1)f(w_2) \in E(H)$,
the proof is analogous by symmetry.
We make a case distinction.
\begin{itemize}
\item \emph{Case 1:} The vertices $w_1$ and $w_2$ were contracted in the first contraction,
that is $w_1 = u$, and $w_2 = v$.
Then because the first symbol of $\cs(G^c,u,v)$ and $\cs(H^c,u',v')$ are the same,
$uv \in E(G)$ if and only if $u'v' \in E(H)$ if and only if $f(u)f(v) \in E(H)$.

\item \emph{Case 2:} The vertices $w_1$ and $w_2$ were part of the same contraction phase.
	Let~$i$ be the iteration such that
	$w_1, w_2 \in \ntw_{G_i}^{T,x,c}(u_{G,i},v_{G,i})$ 
	for some $T,x,c$.
	Then we know $f(w_1), f(w_2) \in \ntw_{H_i}^{T,x',c}(u_{H,i},v_{H,i})$
	for the corresponding vertex~$x'$.
	Note that $w_1x\in E(G_i)$ and $w_2x \in E(G_i)$ holds,
	as otherwise contracting either would create an extra red edge.
	Therefore, $f(w_1)x' \in E(H_i)$ and $f(w_2)x' \in E(H_i)$
	must hold
	in order to have the according atomic types.
	But then contracting either $f(w_1)$ or $f(w_2)$ into $x'$ in $H_i$
	would create a red edge, and therefore they 
	could not have been in $\ntw_{H_i}^{T,x',c}(u_{H,i},v_{H,i})$.
	See 	\Cref{fig:tww1canoncase2}.

	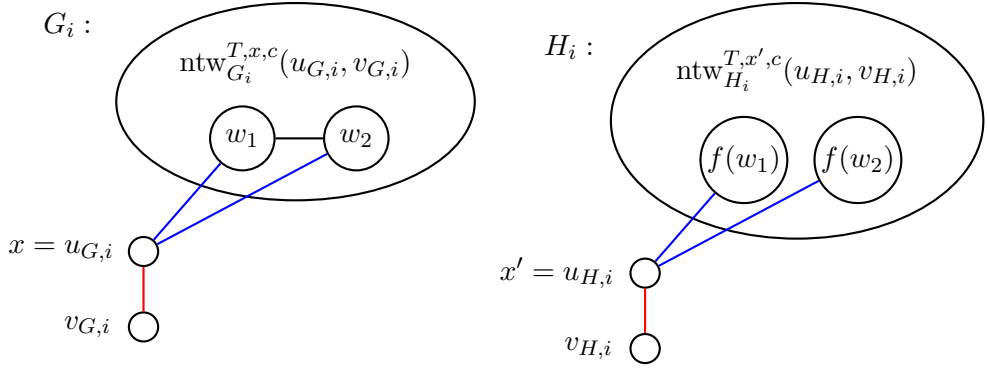
\begin{figure}
			\centering
			\begin{tikzpicture}[minimum size = 1em,
    				baseline=(current bounding box.north),
    				every path/.style={thick}]
    			\node[draw=none] (Gi) at (-1,4) {$G_i:$};
    			\node[draw=none] (ntwleft) at (2,3.5) {$\ntw^{T,x,c}_{G_i}(u_{G,i},v_{G,i})$};
			\begin{scope}[every node/.style={circle,draw}]
   			\node[label=left:{$x = u_{G,i}$}] (ugi) at (0,1) {};
   			\node[label=left:{$v_{G,i}$}] (vgi) at (0,0) {};
   			\node (w1) at (1.3,2.5) {$w_1$};
   			\node (w2) at (2.8,2.5) {$w_2$};
			\end{scope}
   			\node[ellipse,draw=black,inner sep=0,
    				fit = (ntwleft) (w1) (w2)] (ntwleftubble) {};
    			\draw (w1) -- (w2);
    			\draw[draw=blue] (ugi) -- (w1);
    			\draw[draw=blue] (ugi) -- (w2);
    			\draw[draw=red] (ugi) -- (vgi);
		\end{tikzpicture}
		\begin{tikzpicture}[minimum size = 1em,
    			baseline=(current bounding box.north),
    				every path/.style={thick}]
    			\node[draw=none] (Hi) at (-1,4) {$H_i:$};
    			\node[draw=none] (ntwright) at (2,3.7) {$\ntw^{T,x',c}_{H_i}(u_{H,i},v_{H,i})$};
			\begin{scope}[every node/.style={circle,draw}]
   			\node[label=left:{$x' = u_{H,i}$}] (uhi) at (0,1) {};
   			\node[label=left:{$v_{H,i}$}] (vhi) at (0,0) {};
   			\node[inner sep=0.1em] (fw1) at (1.3,2.5) {$f(w_1)$};
   			\node[inner sep=0.1em] (fw2) at (2.8,2.5) {$f(w_2)$};
			\end{scope}
   			\node[ellipse,draw=black,inner sep=0,
    				fit = (ntwright) (fw1) (fw2)] (ntwrightubble) {};
    			\draw[draw=blue] (uhi) -- (fw1);
    			\draw[draw=blue] (uhi) -- (fw2);
    			\draw[draw=red] (uhi) -- (vhi);
		\end{tikzpicture}
    	\caption{Illustration of Case 2 where $x=u_{G,i}$.}
    	\label{fig:tww1canoncase2}
    	\end{figure}
	
\item \emph{Case 3:}
	The vertices $w_1$ and $w_2$ were chosen in different contraction phases. \\
	Let $i \in \N$
	such that 
	$w_1 \in \ntw_{G_i}^{T_1,x_1,c_1}(u_{G,i},v_{G,i})$ 
	for some $T_1,x_1,c_1$,
	and let Let $j \in \N$
	such that 
	$w_2 \in \ntw_{G_j}^{T_2,x_2,c_2}(u_{G,j},v_{G,j})$ 
	for some $T_2,x_2,c_2$.
	Then there are corresponding vertices $x_1'$ and $x_2'$
	such that $f(w_1) \in \ntw_{H_i}^{T_1,x_1',c_1}(u_{H,i},v_{H,i})$ 
	and $f(w_2) \in \ntw_{H_i}^{T_2,x_2',c_2}(u_{H,i},v_{H,i})$.
	By symmetry, 
	let $w_2$ be the vertex that was part of a later contraction phase,
	so either $i < j$,
	or $T_2,x_2,c_2$ appear later than $T_1,x_1,c_1$ in the enumeration
	of the near twin sets.
	Then $w_2x_1 \in E(G_i)$ must hold before contraction,
	as otherwise, contracting $w_1$ into $x_1$ would create an extra red edge in $G_i$.
	Now, assume $f(w_2)x_1' \not \in E(H_i)$.
	\begin{itemize}
		\item \emph{Case 3.1:}
		The vertices $w_1$ and $w_2$ were part of the same iteration,
		that is $i = j$.
		In this case, $f(w_2)$ could never be part of the same contraction
		phase as $w_2$ because $w_2$ is adjacent to $x_1$ and $f(w_2)$ 
		is not adjacent to~$x_1'$.
		See 	\Cref{fig:tww1canoncase31}.

		\begin{figure}[t]
			\centering
			\begin{tikzpicture}[minimum size = 1em,
    				baseline=(current bounding box.north),
    				every path/.style={thick}]
    			\node[draw=none] (Gi) at (-1,4) {$G_i:$};
    			\node[draw=none] (ntw1left) at (2,3.5)
    			{$\ntw^{T_1,x_1,c_1}_{G_i}(u_{G,i},v_{G,i})$};
    			\node[draw=none] (ntw2left) at (2,-2.5)
    			{$\ntw^{T_2,x_2,c_2}_{G_i}(u_{G,i},v_{G,i})$};
			\begin{scope}[every node/.style={circle,draw}]
   			\node[label=left:{$u_{G,i}$}] (ugi) at (0,1) {};
   			\node[label=left:{$x_1 = v_{G,i}$}] (vgi) at (0,0) {};
   			\node (w1) at (2,2.5) {$w_1$};
   			\node (w2) at (2,-1.5) {$w_2$};
			\end{scope}
   			\node[ellipse,draw=black,inner sep=0,
    				fit = (ntw1left) (w1)] (ntwleftubble1) {};
   			\node[ellipse,draw=black,inner sep=0,
    				fit = (ntw2left) (w2)] (ntwleftubble2) {};
    			\draw (w1) -- (w2);
    			\draw[draw=blue] (vgi) -- (w2);
    			\draw[draw=red] (ugi) -- (vgi);
		\end{tikzpicture}
		\begin{tikzpicture}[minimum size = 1em,
    			baseline=(current bounding box.north),
    				every path/.style={thick}]
    			\node[draw=none] (Hi) at (-1,4) {$H_i:$};
    			\node[draw=none] (ntw1left) at (2,3.5)
    			{$\ntw^{T_1,x_1',c_1}_{H_i}(u_{H,i},v_{H,i})$};
    			\node[draw=none] (ntw2left) at (2,-2.5)
    			{$\ntw^{T_2,x_2',c_2}_{H_i}(u_{H,i},v_{H,i})$};
			\begin{scope}[every node/.style={circle,draw}]
   			\node[label=left:{$u_{H,i}$}] (uhi) at (0,1) {};
   			\node[label=left:{$x_1' = v_{H,i}$}] (vhi) at (0,0) {};
   			\node[inner sep=0.1em] (fw1) at (2,2.5) {$f(w_1)$};
   			\node[inner sep=0.1em] (fw2) at (2,-1.5) {$f(w_2)$};
			\end{scope}
   			\node[ellipse,draw=black,inner sep=0,
    				fit = (ntw1left) (fw1)] (ntwleftubble1) {};
   			\node[ellipse,draw=black,inner sep=0,
    				fit = (ntw2left) (fw2)] (ntwleftubble2) {};
    			\draw[draw=blue, dashed] (vhi) -- (fw2);
    			\draw[draw=red] (uhi) -- (vhi);
		\end{tikzpicture}
    		\caption{Illustration of Case 3.1 where $x_1=v_{G,i}$.}
    		\label{fig:tww1canoncase31}
    		\end{figure}
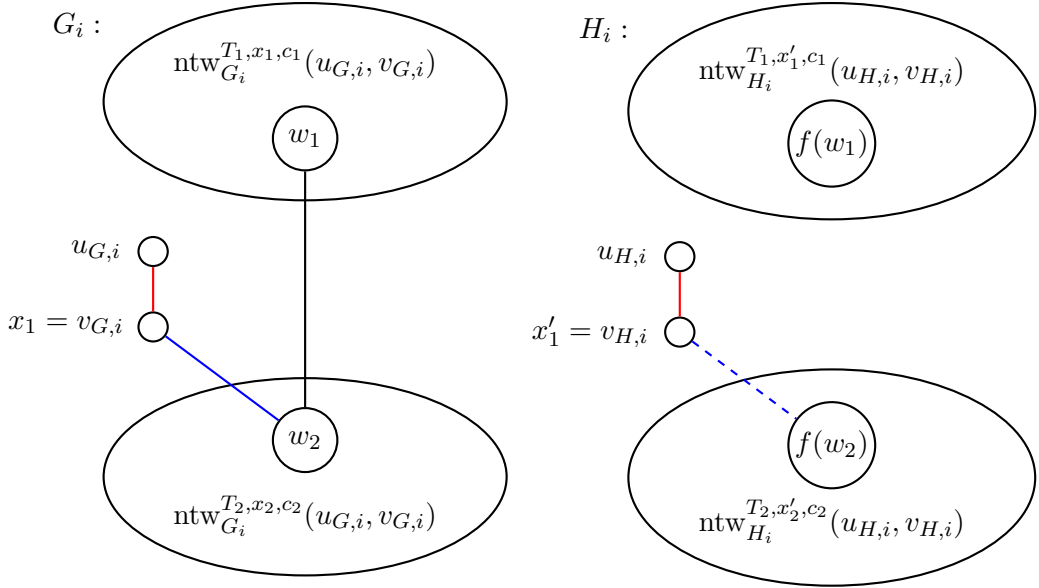

		\item \emph{Case 3.2:} The vertices $w_1$ and $w_2$ were not 
		contracted during the same iteration,
		that is $i < j$ and therefore
		$u_{G,i}$ and $v_{G,i}$ were contracted before~$w_2$.
		Since the contraction did not create a red edge to $w_2$,
		it holds that~$w_2$ is adjacent
		to the vertex resulting from contracting $u_{G,i}$ and $v_{G,i}$
		and~$f(w_2)$ is not connected to the vertex resulting from
		contracting~$u_{H,i}$ and $v_{H,i}$. The argument
		can now be repeated analogously until 
		$w_2$ and~$f(w_2)$ can again not be in the same contraction phase
		because $w_2$ is connected to $u_{G,j}$
		and $f(w_2)$ is not connected to $u_{H,j}$.
		See 	\Cref{fig:tww1canoncase32}.

		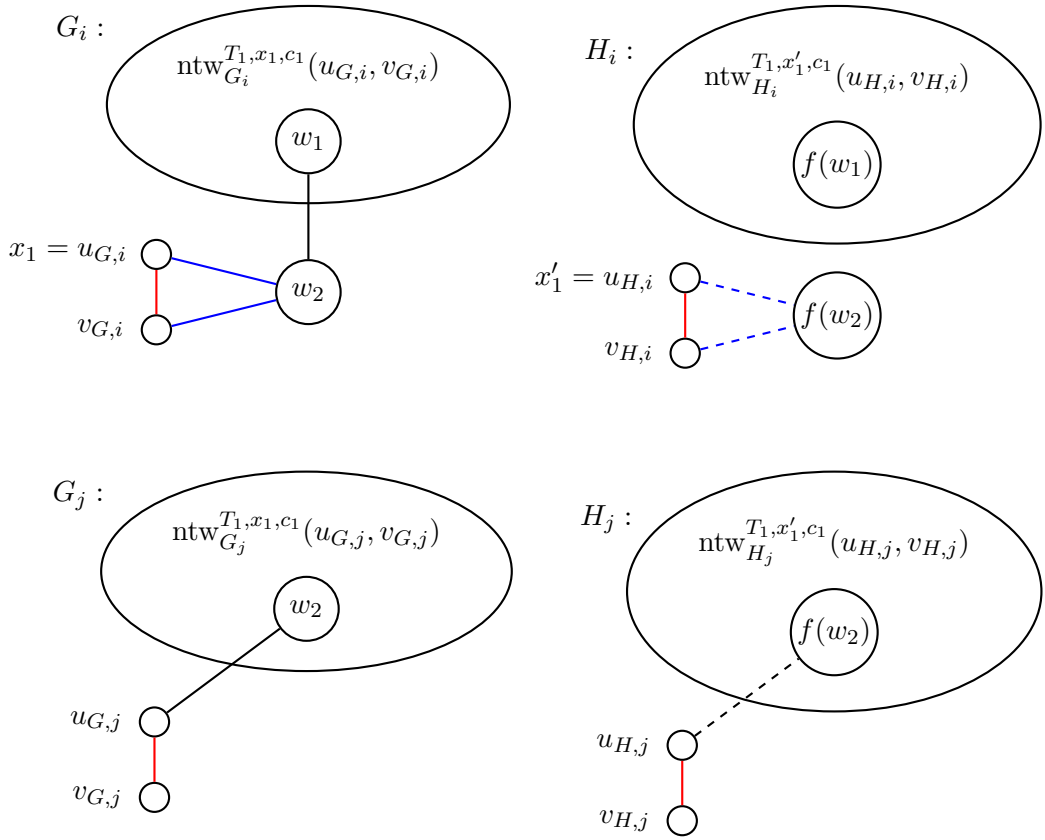
\begin{figure}
			\centering
			\begin{tikzpicture}[minimum size = 1em,
    				baseline=(current bounding box.north),
    				every path/.style={thick}]
    			\node[draw=none] (Gi) at (-1,4) {$G_i:$};
    			\node[draw=none] (ntwleft) at (2,3.5)
    			{$\ntw^{T_1,x_1,c_1}_{G_i}(u_{G,i},v_{G,i})$};
			\begin{scope}[every node/.style={circle,draw}]
   			\node[label=left:{$x_1 = u_{G,i}$}] (ugi) at (0,1) {};
   			\node[label=left:{$v_{G,i}$}] (vgi) at (0,0) {};
   			\node (w1) at (2,2.5) {$w_1$};
   			\node (w2) at (2,0.5) {$w_2$};
			\end{scope}
   			\node[ellipse,draw=black,inner sep=0,
    				fit = (ntwleft) (w1)] (ntwleftubble) {};
    			\draw (w1) -- (w2);
    			\draw[draw=blue] (vgi) -- (w2);
    			\draw[draw=blue] (ugi) -- (w2);
    			\draw[draw=red] (ugi) -- (vgi);
		\end{tikzpicture}
		\begin{tikzpicture}[minimum size = 1em,
    			baseline=(current bounding box.north),
    				every path/.style={thick}]
    			\node[draw=none] (Hi) at (-1,4) {$H_i:$};
    			\node[draw=none] (ntwright) at (2,3.7)
    			{$\ntw^{T_1,x_1',c_1}_{H_i}(u_{H,i},v_{H,i})$};
			\begin{scope}[every node/.style={circle,draw}]
   			\node[label=left:{$x_1' = u_{H,i}$}] (uhi) at (0,1) {};
   			\node[label=left:{$v_{H,i}$}] (vhi) at (0,0) {};
   			\node[inner sep=0.1em] (fw1) at (2,2.5) {$f(w_1)$};
   			\node[inner sep=0.1em] (fw2) at (2,0.5) {$f(w_2)$};
			\end{scope}
   			\node[ellipse,draw=black,inner sep=0,
    				fit = (ntwright) (fw1)] (ntwrightubble) {};
    			\draw[draw=blue, dashed] (vhi) -- (fw2);
    			\draw[draw=blue, dashed] (uhi) -- (fw2);
    			\draw[draw=red] (uhi) -- (vhi);
		\end{tikzpicture}
		
		\vspace*{10mm}		
		
		\hspace*{5mm}
		\begin{tikzpicture}[minimum size = 1em,
    				baseline=(current bounding box.north),
    				every path/.style={thick}]
    			\node[draw=none] (Gj) at (-1,4) {$G_j:$};
    			\node[draw=none] (ntwleft) at (2,3.5)
    			{$\ntw^{T_1,x_1,c_1}_{G_j}(u_{G,j},v_{G,j})$};
			\begin{scope}[every node/.style={circle,draw}]
   			\node[label=left:{$u_{G,j}$}] (ugi) at (0,1) {};
   			\node[label=left:{$v_{G,j}$}] (vgi) at (0,0) {};
   			\node (w2) at (2,2.5) {$w_2$};
			\end{scope}
   			\node[ellipse,draw=black,inner sep=0,
    				fit = (ntwleft) (w2)] (ntwleftubble) {};
    			\draw (ugi) -- (w2);
    			\draw[draw=red] (ugi) -- (vgi);
		\end{tikzpicture}
		\hspace*{5mm}
		\begin{tikzpicture}[minimum size = 1em,
    			baseline=(current bounding box.north),
    				every path/.style={thick}]
    			\node[draw=none] (Hj) at (-1,4) {$H_j:$};
    			\node[draw=none] (ntwright) at (2,3.7)
    			{$\ntw^{T_1,x_1',c_1}_{H_j}(u_{H,j},v_{H,j})$};
			\begin{scope}[every node/.style={circle,draw}]
   			\node[label=left:{$u_{H,j}$}]  (uhi) at (0,1) {};
   			\node[label=left:{$v_{H,j}$}] (vhi) at (0,0) {};
   			\node[inner sep=0.1em] (fw2) at (2,2.5) {$f(w_2)$};
			\end{scope}
   			\node[ellipse,draw=black,inner sep=0,
    				fit = (ntwright) (fw2)] (ntwrightubble) {};
    			\draw[dashed] (uhi) -- (fw2);
    			\draw[draw=red] (uhi) -- (vhi);
		\end{tikzpicture}
    	\caption{Illustration of Case 3.2 where $x_1=u_{G,i}$.}
    	\label{fig:tww1canoncase32}
    	\end{figure}

	\end{itemize}
	Therefore,$f(w_2)x_1' \in E(H_i)$.
	But then $x_1'$ and $f(w_1)$ disagree on $f(w_2)$
	and therefore $f(w_1)$ would not have been chosen for that contraction phase.
\item \emph{Case 4:} For some $i \in \N$,
the vertices
	$u_{G,i}$ and $v_{G,i}$ were contracted, creating a red edge to $w_1$
	or $w_2$. Without loss of generality, let $w_1$ be that vertex.
	Then $w_2v_{G,i+1} \in G_{i+1}$ but $f(w_2)v_{H,i+1} \not \in H_{i+1}$.
	See 	\Cref{fig:tww1canoncase4} for both cases.
	\begin{itemize}
		\item \emph{Case 4.1:} The vertices $u_{G,i+1}$ and $v_{G,i+1}$ are not contracted
		before $w_2$. Then this case is analogous to 3.1
		and $w_2$ and $f(w_2)$ cannot be part of the same
		contraction sequence. 
		\item \emph{Case 4.2:} The vertices $u_{G,i+1}$ and $v_{G,i+1}$ were contracted
		before $w_2$. If there is no $j > i$ such that contracting
		$u_{G,j}$ and $v_{G,j}$ creates a red edge to $w_2$,
		the argument is analogous to Case 3.2.
		If there is,
		then one can see that $w_2u_{G,j} \in G_j$ but $f(w_2)u_{H,j} \not \in H_j$
		and therefore the contraction sequences differ.
		
		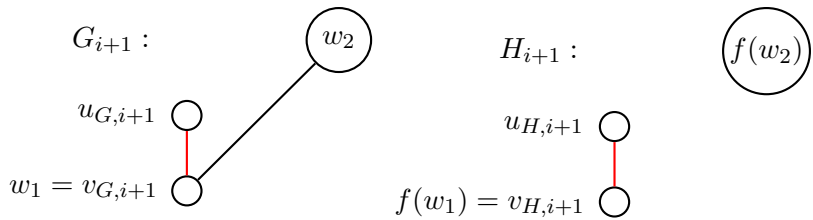
\begin{figure}[t]
			\centering
			\begin{tikzpicture}[minimum size = 1em,
    				baseline=(current bounding box.north),
    				every path/.style={thick}]
    			\node[draw=none] (Gi) at (-1,2) {$G_{i+1}:$};
			\begin{scope}[every node/.style={circle,draw}]
   			\node[label=left:{$u_{G,i+1}$}] (ugi) at (0,1) {};
   			\node[label=left:{$w_1 = v_{G,i+1}$}] (vgi) at (0,0) {};
   			\node (w2) at (2,2) {$w_2$};
			\end{scope}
    			\draw (vgi) -- (w2);
    			\draw[draw=red] (ugi) -- (vgi);
		\end{tikzpicture}
		\begin{tikzpicture}[minimum size = 1em,
    			baseline=(current bounding box.north),
    				every path/.style={thick}]
    			\node[draw=none] (Hi) at (-1,2) {$H_{i+1}:$};
			\begin{scope}[every node/.style={circle,draw}]
   			\node[label=left:{$u_{H,i+1}$}] (uhi) at (0,1) {};
   			\node[label=left:{$f(w_1) = v_{H,i+1}$}] (vhi) at (0,0) {};
   			\node[inner sep=0.1em] (fw2) at (2,2) {$f(w_2)$};
			\end{scope}
    			\draw[draw=red] (uhi) -- (vhi);
		\end{tikzpicture}
	    	\caption{Illustration of Case 4.}
    		\label{fig:tww1canoncase4}
    		\end{figure}

	\end{itemize}
\end{itemize}
Therefore, the bijection $f\colon G^c \rightarrow H^c$ is in fact an isomorphism.
\end{proof}

\section{Definable Canonization for graphs of twin-width 1}\label{tww1:definablecanonization}

For preliminaries on FP+C
as well as a detailed explanation on definable canonization in FP+C,
see \cite{rank-width}.
As in that work, the method here will be described algorithmically instead of directly
within the logic,
but in a way where the implementation in FP+C is clear.
Again, we first consider colored prime graphs.
We first define an edge relation for every starting contraction
pair $u,v$,
along with a number flag indicating if there is a corresponding
contraction sequence, which results in a 5-ary relation:
4 vertex variables and one number variable.
Given a starting contraction pair $u,v$,
initialize the starting contraction accordingly,
the edge set as empty, and the flag as 0.
Start by adding the edge $u,v$ if there is one
in the original graph.
Then, we define the rest in contraction phases,
in parallel to the canonization
algorithm for colored prime graphs.
To accomplish this, keep track of sets $U$ and $V$
along with a contraction phase counter for each,
initialized with $(\{u\},0)$ and $(\{v\},0)$.
If the flag is at 1 in some contraction phase, do nothing. If it is at 0:
\begin{itemize}
\item 
In each even contraction phase $i$,
define the set $W$ of vertices homogeneously adjacent to exactly
one of the sets $U,V$ from the previous contraction phase.
If there is more than one such vertex,
set the flag to 1.
Otherwise, add the corresponding edge
to the edge set.
Add to $U$ all vertices in $V$
and increase the contraction phase counter by 1.
Set $V$ to be the vertex in
$W$ and increase its contraction phase counter by 1.
This is where we need the counters,
as $U$ and $V$ do not monotonically increase
in each step.
\item In each odd contraction phase $i$,
determine the sets of near-twins $W$
with regards to the sets
$U,V$ of the last contraction phase,
for each combination of atomic types and colors.
Since there is already a fixed order on the colors and atomic types
and the ordering within the near-twins sets is irrelevant
as their connections to previously chosen vertices are the same,
this gives us an order on the safely contractable vertices
within this phase. Add them to the corresponding vertex sets
in this order. Increase the contraction phase counters by 1
if there was at least one vertex added.
\end{itemize}
Eventually, this process terminates as either the flag gets set to 1,
or the set $U$ will contain all vertices and the counters
stop increasing. If the flag never gets set to 1,
all edges of the original graph get added in the canonical order
defined in the canonization algorithm for prime graphs of twin width 1
according to the selected starting contraction pair.
To get just one canonical edge set, we iterate
over all the contraction pairs for which the flag does not get set to 1,
and take the lexicographical minimum
of the resulting edge sets.
This is possible in FP+C as in \cite{rank-width}.

From here, one has to lift this result to all graphs
of twin-width 1. To do this, we define the modular
decomposition tree of a graph in FP+C.
To do this, the set of modules is encoded
by a 3-ary relation with two vertex variables and one number
variable: The number variable says which level of the tree
the modules are on, and the corresponding 2-ary relation
is an equivalence relation indicating which vertices
are in the same module.
The edge set is encoded via a 4-ary relation,
two vertex variables and two number variables,
indicating which modules of the different levels of the tree
have an edge between them.
The labelling function is encoded via a 3-ary relation,
one vertex variable and two number variables:
The vertex and first number indicate the module,
and the second number is between 0 and 3 and indicates
the label of the module.
To define the modular decomposition tree using this encoding,
initiate the modules by putting each vertex in its own module
at level 0, labeled single.
Then, calculate the tree bottom up exactly
like in the proof for Lemma \ref{2wl-modules}.
In this way, the top level of the tree defines the maximal modules
containing each vertex and thus also if the graph is prime or not.
Using this, the canonical copy of a graph of twin-width 1
can be defined using the modular decomposition tree
and lexicographical ordering exactly
like in \cref{tww1-canonicalform}.

\section{Sparse graphs of twin-width 2 have bounded tree-width}%Proof of \Cref{thm:sparse:tww2->bounded_tw}}
\label{appendix:sparse:tww2->bounded_tw}
In this section we want to show that the techniques used to prove \Cref{thm:stable:tww2->bounded_rw} and those of \cite{sparse_tww2_bounded_tw} can be combined to show that
the tree-width of twin-width-\(2\) graphs is linearly bounded in the size of their largest semi-induced biclique.
Because bicliques themselves have twin-width \(0\) and tree-width linear in their order, this result is asymptotically optimal.

We start with the analogue of \Cref{lem:stable:existence_of_highly-connected_path} that every graph containing
a large well-linked set must contain a highly connected red path at some point along the partition sequence.

\begin{lemma}\label{lem:sparse:existence_of_highly-connected_path}
Let \(t\geq 0\) be an integer and \(k\coloneqq 7t-6\).
Let \(G\) be a graph without a \(K_{t,t}\)-subgraph which contains an \(\mm\)-well-linked set \(W\) of size \(11k-5\) and
\((\mathcal{P}_i)_{i\in[n]}\) a \(2\)-partition sequence of \(G\).

Then there exists some \(i\in[n]\) and four parts \(X_1,X_2,X_3,X_4\in\mathcal{P}_i\) such that
\begin{itemize}
\item \(X_1X_2X_3X_4\) forms a red \(4\)-path in this order in \(G/\mathcal{P}_i\),
\item there is no red edge \(X_1X_4\),
\item \(\kappa_{G[X_1\cup X_2\cup X_3\cup X_4]}(X_1,X_4)\geq 2t-1\).
\end{itemize}
\begin{proof}
Let \(i\in[n]\) be minimal such that every part of \(\mathcal{P}_i\) contains at most \(2k-1\) vertices of \(W\).
By minimality, \(\mathcal{P}_i\) contains a part \(Z\) which contains at least \(k\) vertices of \(W\).

For some \(d\in\N\) and \(R\in\{=,<,>,\leq,\geq\}\), let \(N^{R d}_{\red}(Z)\) be the set of parts in \(\mathcal{P}_i\) with the specified
distance from \(Z\) in the red graph of \(G/\mathcal{P}_i\).
By abuse of notation, we also denote the set of vertices of \(G\) contained in those parts by \(N^{R d}_\red(Z)\).

Because the red graph has maximum degree at most \(2\), \(N^{\leq d}_{\red}(Z)\) contains at most \(2d+1\) parts
and \(N^{=d}_{\red}(Z)\) contains at most \(2\) parts.
In particular, \(N^{\leq 2}_{\red}\) consists of at most five parts, which together contain at most \(10k-5\) vertices from \(W\).
Thus, the set \(W_3\coloneqq W\cap\bigcup N^{\geq 3}(Z)\) contains at least \(k\) vertices.

Because \(W\) is well-linked, we thus find a collection \(\mathcal{P}\) of \(k\) vertex-disjoint \(Z\)-\(W_3\)-paths in \(G\).
Now, every part of \(P\in \mathcal{P}_i\) which intersects \(W_3\) is either one of the at most two parts in \(N^{=3}_{\red}(Z)\),
or is homogeneously connected to all parts in \(N^{\leq 2}_{\red}(Z)\).

Because \(|Z|\geq k\) and \(G\) is \(K_{t,t}\)-free, there are at most \(t-1\) vertices homogeneously connected to \(Z\).
Similar arguments for the parts in \(N^{\leq 2}_{\red}(Z)\) (with a case distinction on whether they have size at most \(t-1\) or at least \(t\))
yield that at most \(5(t-1)\) vertex-disjoint paths can leave \(N^{\leq 2}_\red(Z)\) through homogeneous connections to other parts.
But since \(|\mathcal{P}|-5(t-1)=k-5t+5=2t-1\), at least one of the at most \(2\) red edges incident to \(N^{\leq 2}_{red}(Z)\)
must contain at least \(t\) of the \(Z\)-\(W_3\)-paths.
Then, the red \(4\)-path starting with this edge and ending at \(Z\) satisfies the claim.
\end{proof}
\end{lemma}

Next, we want to show that the red \(4\)-path whose existence we proved in the previous lemma
forces the existence of a large semi-induced half-graph.
\begin{lemma}\label{lem:sparse:highly-connected_path_implies_biclique}
Let \(G\) be a graph with a \(2\)-partition sequence \((\mathcal{P}_i)_{i\leq n}\) of \(G\),
and assume for some \(m\in[n]\), there exists four parts \(X_1,X_2,X_3,X_4\in\mathcal{P}_m\) such that
\begin{itemize}
\item the four parts \(X_1\), \(X_2\), \(X_3\) and \(X_4\) form a red path in this order in the trigraph \(G/\mathcal{P}_m\),
	and there is no red edge \(X_1X_4\),
\item In \(G[X_1\cup X_2\cup X_3\cup X_4]\), there exist \(3t-2\) vertex-disjoint \(X_1\)-\(X_4\)-paths.
\end{itemize}
Then \(G\) contains a semi-induced biclique \(K_{t,t}\).
\begin{proof}
We first argue that unless \(G\) contains a semi-induced \(K_{t,t}\),
no two of the four parts \(X_1\), \(X_2\), \(X_3\) and \(X_4\) are fully connected.
Indeed, because there exist \(2t-1\) vertex-disjoint \(X_1\)-\(X_4\)-paths, the parts \(X_1\) and \(X_2\) both have size at least \(3t-2\).
Thus, they either share no edge or induce a large biclique, in which case we are done. But if \(X_1\) and \(X_4\) share no edge,
then \(X_2\cup X_3\) is a \(X_1\)-\(X_4\)-separator in \(G[X]\) and thus has size at least \(3t-2\geq 2t-1\). This implies that w.l.o.g. \(|X_3|>t\).
But then, \(G[X_1,X_3]\) either induces a large biclique, in which case we are done, or contains no edge. But then, \(X_2\)
is also a \(X_1\)-\(X_4\)-separator, which implies that also \(|X_2|\geq 2t-1\). If \(G[X_2,X_4]\) does not induce a large biclique,
the parts \(X_2\) and \(X_4\) must thus also be disconnected. In particular, this means that all of the vertex-disjoint \(X_1\)-\(X_4\)-paths
pass through the cut \(X_2\)-\(X_3\).

Now, let \(i>m\) be minimal such that there do not exist four parts \(Y_1,Y_2,Y_3,Y_4\in\mathcal{P}_i\)
with the following properties:
\begin{itemize}
\item the four parts \(Y_1\), \(Y_2\), \(Y_3\), and \(Y_4\) form a red path in this order in the trigraph \(G/\mathcal{P}_i\),
\item \(Y_2\subseteq X_2\) and \(Y_3\subseteq X_3\).
\end{itemize}
Note that such an \(i\) exists, because the graph \(G=G/\mathcal{P}_{|G|}\) contains no red edges
and thus surely does not contain a red \(4\)-path.

Because \(i\) is minimal, we know that \(G/\mathcal{P}_{i-1}\) does contain such a red \(4\)-path
\(Y_1Y_2Y_3Y_4\). In \(G/\mathcal{P}_i\), one of these four parts, say \(Y_i\)
splits into two parts \(Y_i^1\) and \(Y_i^2\) such either \(i\in\{2,3\}\) and \(Y_i^1\) and \(Y_i^2\)
are not connected by a red edge, or one of the red neighbors \(Y_j\) of \(Y_i\)
does not share a red edge with either \(Y_i^1\) or \(Y_i^2\).
In the latter case, we say that the red edge \(Y_iY_j\) is broken,
while in the former case, we say that the part \(Y_i\) is broken.

By symmetry, we may assume that either \(Y_3\) or one of the edges \(Y_2Y_3\) or \(Y_3Y_4\) is broken.

If the red edge \(Y_2Y_3\) is broken, then consider the bipartite graph \(G[X_2,X_3]\).
\begin{claim}
The bipartite graph \(G[X_2,X_3]\) admits a \(1\)-partition sequence whose final two parts are \(X_2\) and \(X_3\).
\begin{proof}
Consider the partition sequence induced by the partition sequence \((\mathcal{P}_j)_{m\leq j\in[n]}\) on \(X_2\cup X_3\).
We claim that at every point along this sequence, there is at most one red edge crossing the cut \((X_2,X_3)\).
Indeed, this is clearly true in \(G/\mathcal{P}_j\) for every \(j\geq i\), because there is no red edge
crossing \((X_2,X_3)\) in \(G/\mathcal{P}_i\) and thus also not in any earlier trigraphs.
Moreover, this is clearly true in \(G/\mathcal{P}_m\), where this single red edge is \(X_2X_3\) itself.
Now, assume the claim is true in \(G/\mathcal{P}_j\) for some \(m\leq j<i\). We show that the claim is also true in \(G/\mathcal{P}_{j+1}\). Assume the partition \(\mathcal{P}_{j+1}\) is obtained from \(\mathcal{P}_j\)
by splitting some part \(P\) into two parts \(P_1\) and \(P_2\). If \(P\) is not incident to a red edge crossing the cut
\((X_2,X_3)\), then neither are \(P_1\) or \(P_2\). Thus, assume \(P\) is incident to the single red edge \(PQ\) crossing the cut, which gets split up into the two red edges \(P_1Q\) and \(P_2Q\). But because \(j<i\), the choice of \(i\)
implies that both \(P\) and \(Q\) are incident to a second red edge in \(G/\mathcal{P}_j\).
But this would imply that \(Q\) has red degree \(3\) in \(G/\mathcal{P}_{j+1}\), which is a contradiction.
\end{proof}
\end{claim}
Because we furthermore already argued that all of the \(3t-2\) vertex-disjoint \(X_1\)-\(X_4\)-paths must cross the cut \((X_2,X_3)\),
the graph \(G[X_2,X_3]\) contains a matching of size \(3t-2\) and thus a semi-induced \(K_{t,t}\) by \Cref{lem:partial_half-graph_has_matchings<=2biclique}.

Thus, we are left with the case that the part \(Y_3\) or the edge \(Y_3Y_4\) is broken.
In the former case, assume \(Y_2\) is incident via a red edge to \(Y_3^1\) but not \(Y_3^2\).
Then, consider instead the bipartite graph \(G[X_2,X_3\setminus Y_3^1]\). Because there is no longer
a red edge crossing the cut \((X_2,X_3\setminus Y_3^1)\) in \(G/\mathcal{P}_i\),
we can argue just as in the previous claim that there exists a \(1\)-partition sequence
of \(G[X_2,X_3\setminus Y_3^1]\) whose final two parts are \(X_2\) and \(X_3\setminus Y_3^1\).
Then, because \(Y_3^1\) is not incident via a red edge to any part within \(X_3\setminus Y_3^1\)
in \(G/\mathcal{P}_i\), at most \(t-1\) vertex-disjoint \(X_1\)-\(X_4\)-paths can pass through \(Y_3^1\).
Thus, we find a matching of size \(2t-1\) in \(G[X_2,X_3\setminus Y_3^1]\) and thus
a semi-induced \(K_{t,t}\) by \Cref{lem:partial_half-graph_has_matchings<=2biclique}.

Finally, if the edge \(Y_3Y_4\) is broken, we consider the bipartite graph \(G[X_2,X_3\setminus Y_3]\).
By the same reasoning as before, this graph has twin-width \(1\) and contains a matching of size \(2t-1\)
and thus a semi-induced \(K_{t,t}\) by \Cref{lem:partial_half-graph_has_matchings<=2biclique}.
\end{proof}
\end{lemma}

Combining the previous two lemmas yields the desired sharpening of the result of \cite{sparse_tww2_bounded_tw},
completely analogously to the proof of \Cref{thm:stable:tww2->bounded_rw}.
\begin{corollary}
Every graph \(G\) with twin-width at most \(2\) which does not contain \(K_{t,t}\) as a subgraph has tree-width at most \(231(t-1)+6\).
\end{corollary}
\end{appendices}

\end{document}